\documentclass[a4,psfig,epsf,11pt]{article}
\usepackage[a4paper,left=2.6cm,right=2.6cm,top=2.3cm,bottom=1.9cm]{geometry}

\usepackage{setspace}
\usepackage{graphicx}
\usepackage[utf8]{inputenc}
\usepackage{amsfonts}
\usepackage{amsmath}
\usepackage{color}
\definecolor{verdemar}{rgb}{0.75,0.78,0.65}
\usepackage[mathcal]{euscript}
\usepackage{url}
\usepackage{subfigure}
\usepackage{psfrag}

\newtheorem{proposicao}{Proposition}[subsection]
\newtheorem{teorema}{Theorem}[subsection]
\newtheorem{lema}{Lemma}[subsection]

\newtheorem{exem}{Example}[subsection]

\newtheorem{obs}{Remark}[subsection]
\newtheorem{Def}{Definition}[subsection]

\newenvironment{proof}{{\noindent\bf Proof.} }
                       {\hfill\rule{2.1mm}{2.1mm} \bigskip }

\newcommand{\R}{\mathbb{R}}
\newcommand{\N}{\mathbb{N}}

\begin{document}
\title{ A Linear Scalarization Proximal Point  Method for Quasiconvex  Multiobjective Minimization }
\author{E. A. Papa Quiroz\thanks{ Mayor de San Marcos National University, Department of Mathematical Sciences,  Lima, Per\'{u} and Federal University of Rio de Janeiro, Computing and Systems Engineering Department, post office box  68511,CEP 21945-970, Rio de Janeiro, Brazil(erikpapa@gmail.com).}
\and{H. C. F. Apolinario\thanks{Federal  University of Tocantins, Undergraduate Computation Sciences Course, ALC NO 14 (109 Norte) AV.NS.15 S/N , CEP 77001-090, Tel: +55 63 8481-5168; +55 63 3232-8027; FAX +55 63 3232-8020, Palmas, Brazil (hellena@uft.edu.br).}}\\
\and{K.D.V. Villacorta\thanks{Federal University of Paraíba, Campus V-Mangabeira, João Pessoa-Paraíba, Brazil. CEP: 58.055-000 }}
\and{P. R. Oliveira\thanks{Federal University of Rio de Janeiro, Computing and Systems Engineering Department, post office box  68511,CEP 21945-970, Rio de Janeiro, Brazil(poliveir@cos.ufrj.br).}}}
\date{\today}
\maketitle
\ \\[-0.5cm]

\begin{center}
{\bf Abstract}
\end{center}
In this paper we propose a linear scalarization proximal point algorithm for solving arbitrary lower semicontinuous quasiconvex multiobjective minimization problems. Under some natural assumptions and using the condition that the proximal parameters are bounded we prove the convergence of the sequence generated by the algorithm and when the objective functions are continuous, we prove the convergence to a generalized critical point. Furthermore, if each iteration minimize the proximal regularized function and the proximal parameters converges to zero we prove the convergence to a weak Pareto solution. In the continuously differentiable case, it is proved the global convergence of the sequence to a Pareto critical point and we introduce an inexact algorithm with the same convergence properties. We also analyze particular cases of the algorithm obtained finite convergence to a Pareto optimal point when the objective functions are convex and a sharp minimum condition is satisfied. 
\\\\
\noindent{\bf Keywords:} Multiobjective minimization, lower semicontinuous quasiconvex functions, proximal point methods, Fejér convergence, Pareto-Clarke critical point, finite convergence.
 
\section{Introduction}
\noindent

In this work we consider the multiobjective minimization problem:
\begin{eqnarray}
\textrm{min}\lbrace F(x): x \in \mathbb{R}^n\rbrace
\label{prob}
\end{eqnarray}
where $F=(F_1,F_2,...,F_m): \mathbb{R}^n\longrightarrow \mathbb{R}^m\cup \{+ \infty \}^m$ is a lower semicontinuous and
quasiconvex vector function on the Euclidean space $ \mathbb{R}^n.$  The above notation means that each $F_i$ is an extended function, that is, $F_i:\mathbb{R}^n\longrightarrow \mathbb{R}^m\cup \{+ \infty \}$ .

The main motivation to study this problem are the consumer demand theory in economy, where the quasiconvexity of the objective vector function is a natural condition associated to diversification of the consumption, see Mas Colell et al. \cite{Colell}, and the quasiconvex optimization models in location theory, see \cite{Gromicho}.

Recently Apolinario et al. \cite{apo} has been introduced an exact linear scalarization proximal point algorithm to solve the above class of problems when the vector function $F$ is locally Lipschitz and $\textnormal{dom}(F)=\R^n$. The proposed iteration was the following:
given $p^{k} \in \mathbb{R}^n$, find $p^{k+1}\in \Omega_k=\left\{ x\in \mathbb{R}^n: F(x) \preceq F(p^k)\right\}$ such that:
 $$
 0 \in \partial^o\left( \left\langle F(.), z_k\right\rangle  + \dfrac{\alpha_k}{2} \Vert\ .\  - p^k \Vert ^2 \right) (p^{k+1}) + \mathcal{N}_{\Omega_k}(p^{k+1})
$$
where ${\partial}^o$ is the Clarke subdifferential, see Subsection 2.5 of \cite{apo}, $\alpha_k > 0 $, $\left\{z_k\right\} \subset \mathbb{R}^m_+\backslash \left\{0\right\}$, $\left\|z_k\right\| = 1$ and $\mathcal{N}_{\Omega_k}(p^{k+1})$ the normal cone to $\Omega_k$ at $x^{k+1},$ see Definition \ref{normal} in Section \ref{Prelimin} of this paper. The authors proved, under some natural assumptions, that the sequence  generated by the above algorithm is well defined and converges globally to a Pareto-Clarke critical point.

Unfortunately, the algorithm proposed in that paper can not be applied to a general class of proper lower semicontinuous quasiconvex functions, and thus can not be applied to solve constrained multiobjective problems nor continuous quasiconvex functions which are not locally Lipschitz. Moreover, for a future implementation and application for example to costly improving behaviors of strongly averse agents in economy (see Sections 5 and 6 of Bento et al. \cite{Bento}), that paper did not provide an inexact version of the proposed algorithm .

Thus we have two motivations in the present paper: the first motivation is to extend the convergence properties of the linear scalarization proximal point method introduced in \cite{apo} to solve more general, probabily constrained, quasiconvex multiobjective problems of the form (\ref{prob}) and the second ones is to introduce an inexact algorithm when $F$ is continuously differentiable on $\mathbb{R}^m$.

The main iteration of the proposed algorithm is: Given $x^k,$ find $x^{k+1} $ such that
 \begin{eqnarray}
 0 \in \hat{\partial}\left( \left\langle F(.), z_k\right\rangle  + \dfrac{\alpha_k}{2} \Vert\ .\  - x^k \Vert ^2  + \delta_{\Omega_k}(.) \right) (x^{k+1}) 
 \label{subdiferencial3i}
 \end{eqnarray}
where $\hat{\partial}$ is the Fréchet subdifferential, see Subsection \ref{frechet},  \ $\Omega_k= \left\{ x\in \mathbb{R}^n: F(x) \preceq F(x^k)\right\}$, $\alpha_k > 0 $, $\left\{z_k\right\} \subset \mathbb{R}^m_+\backslash \left\{0\right\}$ and $\left\|z_k\right\| = 1$.

Some works related to this paper are the following:
\begin{itemize}
	\item Bento et al. \cite{Bento} introduced the nonlinear scalarized proximal iteration:
$$
y^{k+1}\in \arg \min \left\{f\left( F(x)+ \delta_{\Omega_k}(x)e+ \dfrac{\alpha_k}{2} \Vert\ .\  - y^k \Vert ^2e \right): x\in \R^n  \right\}
$$
where $f:\R^n\longrightarrow \R$ is a function defined by $f(y):=\max_{i\in I}\{ \langle y,e_i \rangle \}$
with $e_i$ is the canonical base of the space $\R^n,$ $\Omega_k= \left\{ x\in \mathbb{R}^n: F(x) \preceq F(y^k)\right\}$ and $e=(1,1,...,1)\in \R^n.$ Assuming that $F:\mathbb{R}^n\longrightarrow \mathbb{R}^m$ is quasiconvex and continuously differentiable and under some natural assumptions the authors proved that the sequence $\{y^k\}$ converges to a Pareto Critical point of $F.$ Furthermore, assuming that $F$ is convex, the weak Pareto optimal set is weak sharp for the multiobjective problem and that the sequence is generated by the following unconstrained iteration
$$
y^{k+1}:=\arg \min \left\{f(F(x)) + \dfrac{\alpha_k}{2} \Vert\ x\  - y^k \Vert ^2:x\in \mathbb{R}^n  \right\}
$$  
then the above iteration obtain a Pareto optimal point after a finite number de iterations.

The difference between our work and the paper of Bento et al., \cite{Bento}, is that in the present paper we consider a linear scalararization of $F$ instead of a nonlinear ones proposed in \cite{Bento}, another difference is that our assumptions are more weak, in particular, we obtain convergence results for nondifferentiable quasiconvex functions.

\item Makela et al., \cite{makela}, developed a multiobjective proximal bundle method for nonsmooth optimization where the objective functions are locally Lipschitz (not necessarily smooth nor convex). The proximal method is not directly based on employing any scalarizing function but based on a improvement function $H:\mathbb{R}^n\times \mathbb{R}^n \longrightarrow \mathbb{R}$ defined by $H(x,y)=\max \{ F_i(x)-F_i(y),g_j(x):i=1,...,m, j=1,...,r \}$ with $\textnormal{dom} F=\{x\in \mathbb{R}^n: g_j(x)\geq 0, j=1,...,r\}.$ If $F_i$ and $g_j$ are pseudoconvex and weakly semismooth functions and certain constraint qualification is valid, the authors proved that any accumulation point of the sequence is a weak Pareto solution and without the assumption of pseudoconvex, they obtained that any accumulation point is a  substationary point, that is, $0\in \partial H(\bar x, \bar x),$ where $\bar x$ is an accumulation point.
\item Chuong et al., \cite{chuong}, developed three algorithms of the so-called hybrid approximate proximal type to find Pareto optimal points for general class of convex constrained problems of vector optimization in finite and infinite dimensional spaces, that is, $\min_C \{F(x): x\in \Omega\},$ where $C$ is a closed convex and pointed cone and the minimization is understood with respect to the ordering relation given by $y\preceq_{C}x$ if and only if $x-y\in C$. Assuming that the set $\left(F(x^0) - C\right) \cap F(\Omega)$ is $C$ - quasi-complete for $\Omega,$ that is, for any sequence $\{u_l\}\subset \Omega$ with $u_0=x_0$ such that $F(u_{l+1})\preceq_{C}F(u_l)$ there exists $u\in VI(\Omega, A)$ satisfying $F(u)\preceq_{C}F(u_l),$ for every $l\in \N;$ and the assumption that $F$ is $C^{+}-$ uniformly semicontinuous on $\Omega,$ the authors proved the convergence of the sequence generates by its algorithm.
\end{itemize}

Under the assumption that $F$ is a proper lower semicontinuous quasiconvex vector function and the assumption that the set $\left(F(x^0) - \mathbb{R}^m_+\right)\cap F(\mathbb{R}^n)$ is $\mathbb{R}^m_+$ - complete we prove the global convergence of the sequence $\{x^k\},$ generated by (\ref{subdiferencial3i}), to the set
$$E = \left\{x \in \mathbb{R}^n: F\left(x\right)\preceq F\left(x^k\right),\ \ \forall\  k \in \mathbb{N}\right\}.$$
Additionally, if $F:\mathbb{R}^n\longrightarrow \mathbb{R}^m $ is continuous, and $0<\alpha_k<\bar{\alpha},$ for some $\bar{\alpha}>0,$ we prove that
$\lim  \limits_{k\rightarrow +\infty}g^{k} = 0,$ where $g^k \in \hat{\partial}\left( \left\langle F(.), z_k\right\rangle + \delta_{\Omega_k}\right)(x^{k+1}).$ In the particular case when $ \lim \limits_{k\rightarrow +\infty}\alpha_k= 0$ and the iterations are given by 
\begin{eqnarray}
  x^{k+1}\in \textnormal{arg min} \left\{\left\langle F(x), z_k\right\rangle+\frac{\alpha_k}{2}\left\|x - x^k\right\|^2 : x\in\Omega_k\right\},
  \label{recursao0i}
 \end{eqnarray}
then the sequence $\lbrace x^k\rbrace$ converges to a weak pareto solution of the problem $(\ref{prob})$.

When the vector function $F: \mathbb{R}^n\longrightarrow \mathbb{R}^m$ is continuously differentiable and $0<\alpha_k<\bar{\alpha},$ for some $\bar{\alpha}>0,$ we prove that the sequence $\{x^k\},$ generated by (\ref{subdiferencial3i}), converges to a Pareto critical point of the problem (\ref{prob}). Then, we introduce an inexact proximal algorithm given by
\begin{equation}
 0 \in \hat{\partial}_{\epsilon_k} \left( \langle F(x), z_k \rangle \right) (x) + \alpha_k\left(x - x^k\right) + \mathcal{N}_{\Omega_k}(x),
 \label{diferenciali}
 \end{equation}
  \begin{equation}
 \label{deltai}
\displaystyle \sum_{k=1}^{\infty} \delta_k < \infty,
\end{equation}
where $\delta_k = \textnormal{max}\left\lbrace  \dfrac{\varepsilon_k}{\alpha_k}, \dfrac{\Vert \nu_k\Vert}{\alpha_k}\right\rbrace,$ $\varepsilon_k\geq 0,$  and $\hat{\partial}_{\varepsilon_k}$ is the Fréchet $\varepsilon_k$-subdifferential.
We prove the convergence of $\{x^k\},$ generated by (\ref{diferenciali}) and (\ref{deltai}) to a Pareto critical point of the problem (\ref{prob}).

We also analyze some conditions to obtain finite convergence of a particular case of the proposed algorithm. 

The paper is organized as follows:  In  Section 2 we recall some concepts and basic results on multiobjective optimization, descent direction, scalar representation, quasiconvex and convex functions, Fr\'echet and Limiting subdiferential,  $\epsilon-$Subdifferential and Fej\'{e}r convergence. In Section 3 we present the problem and we give an example of a quasiconvex model in demand theory. In Section 4 we introduce an exact algorithm and analyze its convergence. In Section 5 we present an inexact algorithm for the differentiable case and analyze its convergence. In Section 6, we introduce an inexact algorithm for nonsmooth proper lower semicontinuous convex multiobjective minimization and using some concepts of weak sharp minimum we prove the convergence of the iterations in a finite number of steps to a Pareto optimal point. In Section 7 give a numerical example of the algorithm and in Section 8 we give our conclusions. 

\section{Preliminaries}
\label{Prelimin}
In this section, we present some basic concepts and results that are of fundamental importance for the development of our work. These facts can be found, for example, in Hadjisavvas \cite{Had}, Mordukhovich \cite{Mordukhovich} and, Rockafellar and Wets \cite{Rockafellar}.

\subsection{Definitions, notations and some basic results}

Along this paper  $ \mathbb{R}^n$ denotes an Euclidean space, that is, a real vectorial space with the canonical inner product $\langle x,y\rangle=\sum\limits_{i=1}^{n} x_iy_i$ and the norm given by $||x||=\sqrt{\langle x, x\rangle }$.\\
Given a function {\small $f :\mathbb{R}^n\longrightarrow \mathbb{R}\cup\left\{+\infty\right\}$}, we
denote by $\textnormal{dom}(f)= \left\{x \in \mathbb{R}^n: f(x) < + \infty \right\},$ the {\it effective domain } of $f$.
If $\textnormal{dom}(f) \neq \emptyset$, $f $ is called proper.
If {\footnotesize $\lim  \limits_{\left\|x\right\|\rightarrow +\infty}f(x) = +\infty$}, $f$ is called coercive.  We denote by arg min $\left\{f(x): x \in \mathbb{R}^n\right\}$ the set of minimizer of the function $f$ and by  $f * $, the optimal value of problem: $\min \left\{f(x): x \in \mathbb{R}^n\right\},$  if it exists.
The \  function \ $f$ is {\it lower semicontinuous} at $\bar{x}$ if for all sequence $\left\{x_k\right\}_{k \in \mathbb{N}} $ such that $\lim  \limits_{k \rightarrow +\infty}x_k = \bar{x}$ we obtain that $f(\bar{x}) \leq \liminf \limits_{k \rightarrow +\infty}f(x_k)$.\\
The  next result ensures that the set of minimizers of a function, under some assumptions, is nonempty.
\begin{proposicao}{\bf (Rockafellar and Wets \cite{Rockafellar}, Theorem 1.9)}\\
Suppose that {\small $f:\mathbb{R}^n\longrightarrow \mathbb{R}\cup\left\{+\infty\right\}$} is {\it proper, lower semicontinuous} and coercive, then the optimal value  $ f^*$\ is finite and the set $\textnormal{arg min}$ $\left\{f(x): x \in \mathbb{R}^n\right\}$ is nonempty and compact.
\label{coercivaesemicont}
\end{proposicao}
\begin{Def}
Let $D \subset \mathbb{R}^n$ and $\bar{x} \in D$.   The normal cone to $D$ at $\bar{x} \in D$ is given by $\mathcal{N}_{D}(\bar{x}) = \left\{v \in \mathbb{R}^n: \langle  v,  x - \bar{x}\rangle \leq 0, \forall \ x \in D\right\}$.
\label{normal}
\end{Def}
It follows an important result that involves sequences of non-negative numbers which will be useful in Section  5.
\begin{lema}
Let $\{w_k\}$, $\{p_k\}$ and $\{q_k\}$ sequences  of non-negative  real numbers. If
\begin{equation*}
w_{k+1} \leq \left( 1 + p_k\right)w_k + q_k, \ \ \ \ \displaystyle \sum_{i=1}^{\infty} p_k < +\infty \ \ \textnormal{and} \ \  \displaystyle \sum_{i=1}^{\infty} q_k < +\infty,
\end{equation*}
then the sequence $\{w_k\}$ is convergent.
\label{p}
\end{lema}
\begin{proof}
See Polyak \cite{Polyak}, Lema 2.2.2.
\end{proof}
 
\subsection{Multiobjective optimization}

In this subsection we present some properties and notation  on multiobjective optimization. Those basic facts can be seen, for example, in  Miettinen \cite{Kaisa} and Luc \cite{Luc}.\\
Throughout this paper we consider the cone $\mathbb{R}^m_+ = \{ y\in \mathbb{R}^m : y_i\geq0, \forall \  i = 1, ... , m \}$, which induce a partial order $\preceq$ in $\mathbb{R}^m$ given by, for $y,y'\in \mathbb{R}^m$,
$y\ \preceq\ y'$ if, and only if, $ y'\ - \ y$  $ \in \mathbb{R}^m_+$, this means that $ y_i \leq \ y'_i,$ for all $ i= 1,2,...,m $ .   Given $ \mathbb{R}^m_{++}= \{ y\in \mathbb{R}^m : y_i>0, \forall \  i = 1, ... , m \}$ the above relation induce the following one $\prec$, induced  by the interior of this cone,  given by, $y\ \prec\ y'$, if, and only if, $ y'\ - \ y$  $ \in \mathbb{R}^m_{++}$, this means that $ y_i < \ y'_i$ for all $ i= 1,2,...,m$.  Those partial orders establish a class of problems known in the literature as Multiobjective  Optimization.\\ \\
Let us consider the unconstrained multiobjective optimization problem (MOP) :
\begin{eqnarray}
 \textrm{min} \left\{G(x): x \in \mathbb{R}^n \right\}
  \label{POM}
\end{eqnarray}
where $G:\mathbb{R}^n\longrightarrow \mathbb{R}^m\cup \{+ \infty \}^m$, with $G = \left(G_1, G_2, ... , G_m\right)$ and $G_i:\mathbb{R}^n \longrightarrow \R, \forall i=1,...,m$.
\begin{Def} {\bf (Miettinen \cite{Kaisa}, Definition 2.2.1)}
 A point $x^* \in \mathbb{R}^n$ is a Pareto optimal point or Pareto solution of the problem $\left(\ref{POM}\right)$, if there does not exist $x \in  \mathbb{R}^n $ such that $ G_{i}(x) \leq G_{i}(x^*)$, for all $i \in \left\{1,...,m\right\}$ and $ G_{j}(x) <  G_{j}(x^*)$, for at least one index $ j \in \left\{1,...,m\right\}$ .
\end{Def}
\begin{Def}{\bf (Miettinen \cite{Kaisa},Definition 2.5.1)}
 A point $x^* \in \mathbb{R}^n$ is a weak Pareto solution of the problem $\left(\ref{POM}\right)$, if there does not exist $x \in  \mathbb{R}^n $ such that $ G_{i}(x) < G_{i}(x^*)$, for all $i \in \left\{1,...,m\right\}$.
\end{Def}
We denote by arg min$\left\{G(x):x\in \mathbb{R}^n \right\}$ and by arg min$_w$ $\left\{G(x):x\in \mathbb{R}^n \right\}$ the set of Pareto solutions and weak Pareto solutions to the problem $\left(\ref{POM}\right)$, respectively.  It is easy to check that\\ arg min$\left\{G(x):x\in \mathbb{R}^n \right\} \subset$ arg min$_w$ $\left\{G(x):x\in \mathbb{R}^n \right\}$.
%

\subsection{Pareto critical point and descent direction}

Let $G:\mathbb{R}^n\longrightarrow \mathbb{R}^m$ be a differentiable function and $x \in \mathbb{R}^n$, the jacobian of $G$ at $x$, denoted by $JG(x)$, is a matrix of order $m \times n$ whose entries  are defined by $\left(JG(x)\right)_{i,j} = \frac{\partial G_i}{\partial x_j}(x)$.  We may represent it by,
\begin{center}
$JG\left(x\right) := \left[\nabla G_1(x) \nabla G_2(x)... \nabla G_m(x)  \right]^T$, $x \in \mathbb{R}^n$.
\end{center}
The image of the jacobian of $G$ at $x$ we denote by 
\begin{center}
$\footnotesize{Im \left(JG\left(x\right)\right) := \lbrace JG\left(x\right)v = \left(\langle \nabla G_1(x) , v\rangle, \langle \nabla G_2(x) , v\rangle, ..., \langle \nabla G_m(x) , v\rangle\right): v \in \mathbb{R}^n \rbrace }$.
\end{center}
A necessary but not sufficient first order optimality condition for the problem 
$(\ref {POM})$ at $x \in \mathbb {R}^n $, is
\begin{eqnarray}
Im \left(JG\left(x\right)\right)\cap\left(-\mathbb{R}^m_{++}\right)=\emptyset.
\label{cond}
\end{eqnarray}
Equivalently, $\forall \ v \in \mathbb{R}^n$, there exists $i_0 = i_0(v) \in \lbrace 1,...,m\rbrace$ such that
\begin{center}
$\langle \nabla G_{i_0}(x) ,  v \rangle \geq 0$.
\end{center}
\begin{Def}
Let $G:\mathbb{R}^n\longrightarrow \mathbb{R}^m$ be a differentiable function. A point $x^* \in \mathbb{R}^n$ satisfying $(\ref {cond})$ is called a Pareto critical point.
\end{Def}
Follows from the previous definition, if a point $x$ is not Pareto critical point, then there exists a direction 
 $v \in \mathbb {R}^ n$ satisfying
\begin{center}
$JG\left(x\right)v \in \left(-\mathbb{R}^m_{++}\right)$,
\end{center}
i.e, $\langle \nabla G_i( x ) , v \rangle < 0, \ \forall \ i \in \lbrace 1,..., m \rbrace$.  As $G$ is continuously differentiable, then
\begin{center}
$\displaystyle \lim_{t \rightarrow 0}\dfrac{G_i(x + tv) - G_i(x)}{t}= \langle \nabla G_i(x) ,v\rangle < 0, \ \forall \ i \in \lbrace 1,..., m \rbrace $.
\end{center}
This implies that $v$ is a {\it descent direction} for the function $G_i$, i.e, there exists $\varepsilon > 0 $, such that
\begin{center}
 $G_i(x + tv)  < G_i(x), \forall \ t \in (0 , \varepsilon ], \forall \ i \in \lbrace 1,..., m \rbrace $. 
 \end{center}
Therefore, $v$ is a {\it descent direction} for $G$ at $x$, i.e, there exists $ \varepsilon > 0 $ such that
 \begin{center}
 $ G(x + tv) \prec G(x), \ \forall \ t \in (0 , \varepsilon]$.
 \end{center}

\subsection{Scalar representation}
In this subsection we present a useful technique in multiobjective optimization which allows to replace the original optimization problem into a scalar optimization problem or a family of scalar problems.
 \begin{Def}{\bf (Luc \cite{Luc}, Definição 2.1)}
A function
$f:\mathbb{R}^n\longrightarrow \mathbb{R}\cup \{+ \infty \}$ is said to be a  strict scalar representation of a map $F:\mathbb{R}^n\longrightarrow \mathbb{R}^m\cup \{+ \infty \}^m$ when given $x,\bar{x}\in \mathbb{R}^n :$
 \begin{center}
$F(x)\preceq F(\bar{x}) \Longrightarrow f(x)\leq f(\bar{x})$\ \ and \ \
$F(x)\prec F(\bar{x}) \Longrightarrow f(x)<f(\bar{x}).$
\end{center}
Furthermore, we say that $f$ is a weak scalar representation of $F$ if
\begin{center}
$F(x)\prec F(\bar{x})\Longrightarrow f(x)<f(\bar{x}).$
\end{center}
\label{escalarizacao} 
\end{Def}
\begin{proposicao}
\label{rep}
Let $f:\mathbb{R}^n\longrightarrow \mathbb{R}\cup \{+ \infty \}$ be a proper function.  Then $f$ is a strict scalar representation of $F$ if, and only if, there exists a strictly increasing function  $g:F\left(\mathbb{R}^n\right)\longrightarrow \mathbb{R}$ such that $f = g \circ F.$
\end{proposicao}
\begin{proof}
See Luc \cite{Luc}, Proposition 2.3.
\end{proof}

\begin{proposicao}
\label{inclusao}
Let $f:\mathbb{R}^n\longrightarrow \mathbb{R}\cup \{+ \infty \}$ be a weak scalar representation of a vector function $F:\mathbb{R}^n\longrightarrow \mathbb{R}^m\cup \{+ \infty \}^m$ and $\textnormal{argmin}\left\{f(x):x\in \mathbb{R}^n\right\}$ the set of minimizer points of $f$. Then, we have
\begin{equation*}
\textnormal{argmin}\left\{f(x): x\in \mathbb{R}^n\right\}\subseteq \textnormal{argmin}_w \{F(x):x\in \mathbb{R}^n\}.
\end{equation*} 
\end{proposicao}
\begin{proof}
It is immediate.
\end{proof}

\subsection{Quasiconvex and Convex Functions}
In this subsection we present the concept and characterization of quasiconvex functions and quasiconvex  multiobjective function. This theory can be found in Bazaraa et al. \cite{Bazaraa}, Luc \cite{Luc}, Mangasarian \cite{Mangasarian}, and references therein.
\begin{Def}
Let $f:\mathbb{R}^n\longrightarrow \mathbb{R} \cup \{+ \infty \}$ be a proper function.  Then, f is
called quasiconvex if for all $x,y\in \mathbb{R}^n$, and for all $ t \in \left[0,1\right]$, it holds
that $f(tx + (1-t) y)\leq \textnormal{max}\left\{f(x),f(y)\right\}$.
\end{Def}
\begin{Def}
Let $f:\mathbb{R}^n\longrightarrow \mathbb{R} \cup \{+ \infty \}$ be a proper function.  Then, f is
called convex if for all $x,y\in \mathbb{R}^n$, and for all $ t \in \left[0,1\right]$, it holds
that $f(tx + (1-t) y)\leq tf(x) + (1 - t)f(y)$.
\end{Def}
Observe that if $f$ is a quasiconvex function then $\textnormal{dom}(f)$ is a convex set. On the other hand, while a convex function can be characterized by the convexity of its epigraph, a quasiconvex function can
be characterized by the convexity of the lower level sets:
\begin{Def}
Let \ \ $F= (F_1,...,F_m):\mathbb{R}^n\longrightarrow \mathbb{R}^m\cup \{+ \infty \}^m$ be a function, then $F$ is $\mathbb{R}^m_+$ -  quasiconvex if every component function of $F$, $F_i: \mathbb{R}^n\longrightarrow \mathbb{R}\cup \{+ \infty \}$, is quasiconvex.
\end{Def}

\begin{Def}
Let \ \ $F= (F_1,...,F_m):\mathbb{R}^n\longrightarrow \mathbb{R}^m\cup \{+ \infty \}^m$ be a function, then $F$ is $\mathbb{R}^m_+$ -  convex if every component function of $F$, $F_i: \mathbb{R}^n\longrightarrow \mathbb{R}\cup \{+ \infty \}$, is convex.
\end{Def}
\subsection{Fréchet  and Limiting Subdifferentials}
\label{frechet}
\begin{Def}
 Let $f: \mathbb{R}^n \rightarrow \mathbb{R} \cup \{ +\infty \}$ be a proper function.
 \begin{enumerate}
 \item [(a)]For each $x \in \textnormal{dom}(f)$, the set of regular subgradients (also called Fréchet subdifferential) of $f$ at $x$, denoted by $\hat{\partial}f(x)$, is the set of vectors $v \in \mathbb{R}^n$ such that
\begin{center}
$f(y) \geq f(x) + \left\langle v,y-x\right\rangle + o(\left\|y - x\right\|)$, where $\lim \limits_{y \rightarrow x}\frac{o(\left\|y - x\right\|)}{\left\|y - x\right\|} =0$.
\end{center}
Or  equivalently, $\hat{\partial}f (x) := \left\{ v \in \mathbb{R}^n : \liminf \limits_{y\neq x,\ y \rightarrow x} \dfrac{f(y)- f(x)- \langle v , y - x\rangle}{\lVert y - x \rVert} \geq 0 \right \}$.\\  If $x \notin \textnormal{dom}(f)$ then $\hat{\partial}f(x) = \emptyset$.
\item [(b)]The set of general subgradients (also called limiting subdifferential) $f$ at $x \in \mathbb{R}^n$, denoted by $\partial f(x)$, is defined as follows:
\begin{center}
$\partial f(x) := \left\{ v \in \mathbb{R}^n : \exists\  x_l \rightarrow x, \ \ f(x_l) \rightarrow f(x), \ \ v_l \in \hat{\partial} f(x_l)\  \textnormal{and}\  v_l \rightarrow v \right \}$.
\end{center}
\end{enumerate}
\label{fech}
\end{Def}
\begin{proposicao}{\bf (Fermat’s rule generalized)}
If a proper function $f: \mathbb{R}^n \rightarrow \mathbb{R} \cup \{+ \infty \}$ has a local minimum at $\bar{x} \in \textnormal{dom}(f)$, then $0\in \hat{\partial} f\left(\bar{x}\right)$.
\label{otimo}
\end{proposicao}
\begin{proof}
See Rockafellar and Wets \cite{Rockafellar}, Theorem 10.1.
\end{proof}

\begin{proposicao}
Let $f: \mathbb{R}^n \rightarrow \mathbb{R} \cup \{+ \infty \}$ be a proper function.  Then, the following properties are true
\begin{enumerate}
\item[(i)]$\hat{\partial}f(x) \subset \partial f(x)$, for all $x \in \mathbb{R}^n$.
\item[(ii)]If $f$ is differentiable at $\bar{x}$ then $\hat{\partial}f(\bar{x}) = \{\nabla f (\bar{x})\}$, so $\nabla f (\bar{x})\in \partial f(\bar{x})$.
\item[(iii)] If $f$ is continuously differentiable in a neighborhood of $x$, then $\hat{\partial}f(x) = \partial f(x) = \{\nabla f (x)\}$.
\item[(iv)] If \ $ g = f + h $ with $f$ finite at $\bar{x}$ and $h$ is continuously differentiable in a neighborhood of $\bar{x}$, then $\hat{\partial}g(\bar{x}) = \hat{\partial}f(\bar{x}) + \nabla h(\bar{x})$ and\
$\partial g(\bar{x}) = \partial f(\bar{x}) + \nabla h(\bar{x})$.
\end{enumerate}
\label{somafinita}
\end{proposicao}
\begin{proof}
See Rockafellar and Wets \cite{Rockafellar}, Exercise $8.8$, page $304$.
\end{proof}


\subsection{{\bf$\varepsilon$}-Subdiffential}
We present some important concepts and results on $\varepsilon$-subdifferential.  The theory of these facts  can be found, for example, in Jofre et  al. \cite{Jofre} and Rockafellar and Wets \cite {Rockafellar}.
\begin{Def}
Let $f: \mathbb{R}^n \rightarrow \mathbb{R} \cup \{+ \infty \}$ be a proper lower semicontinuous function and let $\varepsilon$ be an arbitrary nonnegative real number. The Fréchet $\varepsilon$-subdifferential of $f$ at $x \in \textnormal{dom}(f)$ is defined by
\begin{eqnarray}
\hat{\partial}_{\varepsilon}f (x) := \left\{ x^* \in \mathbb{R}^n : \liminf \limits_{\lVert h \rVert  \rightarrow 0} \dfrac{f(x + h)- f(x)- \langle x^* , h\rangle}{\lVert h \rVert} \geq - \varepsilon \right \}
\label{frechet1}
\end{eqnarray}
\label{frechet3}
\end{Def}
\begin{obs}
When $\varepsilon = 0$, $(\ref{frechet1})$ reduces to the well known Fréchet subdifferential, wich is denoted by $\hat{\partial}f(x)$, according to \textnormal{Definition} $\ref{fech}$.  More precisely, 
\begin{center}
$x^* \in \hat{\partial}f(x)$, if and only if, for each $\eta > 0 $ there exists $\delta > 0$ such that\\
$\langle x^* , y - x\rangle \leq f(y) - f(x) + \eta\Vert y - x \Vert$, for all $y \in x + \delta \textnormal{B}$,
\end{center}
where $B$ is the closed unit ball in $\R^n$ centered at zero. Therefore $\hat{\partial}f(x)=\hat{\partial}_{0}f(x)\subset\hat{\partial}_{\varepsilon}f(x).$
\label{frechet2}
\end{obs}
From Definition 5.1 of Treiman, \cite{Treiman},
\begin{center}
$x^* \in \hat{\partial}_{\epsilon}f (x)\Leftrightarrow x^* \in \hat{\partial}(f + \epsilon\Vert . - x\Vert)(x)$.
\end{center}
Equivalently, $x^* \in \hat{\partial}_{\epsilon}f (x)$, if and only if, for each $\eta > 0$, there exists $\delta > 0$ such that 
\begin{center}
$\langle x^* , y - x\rangle \leq f(y) - f(x) + (\epsilon + \eta)\Vert y - x \Vert$, for all $y \in x + \delta \textnormal{B}$.
\end{center}
\vspace{0,2cm}
We now defined a new kind of approximate subdifferential.
\begin{Def}
The limiting Fréchet $\varepsilon$-subdifferential of $f$ at $x \in \textnormal{dom} (f)$ is defined by 
\begin{eqnarray}
\partial_\varepsilon f(x) := \limsup \limits_{y \stackrel{f}{\longrightarrow} x} \hat{\partial}_{\varepsilon}f (y)
\end{eqnarray}
where $$\limsup \limits_{y \stackrel{f}{\longrightarrow} x} \hat{\partial}_{\varepsilon}f (y):=\lbrace x^* \in \mathbb{R}^n: \exists\   x_l \longrightarrow x, f(x_l)\longrightarrow f(x), x^*_l \longrightarrow x^* \ \textnormal{with}\  x^*_l \in \hat{\partial}_{\varepsilon}f (x_l)\,  \rbrace$$
\end{Def}
\vspace{0,5cm}
In the case where $f$ is continuously differentiable, the limiting Fréchet $\varepsilon$-subdifferential takes a very simple form, according to the following proposition
\begin{proposicao}
Let $f: \mathbb{R}^n \rightarrow \mathbb{R}$ be a continuously differentiable function at $x$ with derivative 
$\nabla f (x)$.  Then
\begin{center}
$\partial_\varepsilon f(x) = \nabla f (x) + \varepsilon B$.
\end{center}
\label{fdif}
\end{proposicao}
\begin{proof}
See Jofré et al., \cite{Jofre}, Proposition 2.8.
\end{proof}


\subsection{Fejér convergence}
\begin{Def}
A seguence $\left\{y_k\right\} \subset \mathbb{R}^n$ is said to be Fejér convergent to a set $U\subseteq \mathbb{R}^n$ if,
$\left\|y_{k+1} - u \right\|\leq\left\|y_k - u\right\|, \forall \ k \in \mathbb{N},\ \forall \ u \in U$.
\end{Def}
The following result on Fejér convergence is well known.
\begin{lema}
If $\left\{y_k\right\}\subset \mathbb{R}^n$ is Fejér convergent to some set $U\neq \emptyset$, then:
\begin{enumerate}
\item [(i)]The sequence $\left\{y_k\right\}$ is bounded.
\item [(ii)]If an accumulation point $y$ of $\left\{y_k\right\}$ belongs to $ U$, then $\lim  \limits_{k\rightarrow +\infty}y_k = y$.
 \end{enumerate}
\label{fejerlim1}
\end{lema}
\begin{proof}
See Schott \cite{Schott}, Theorem $2.7$.
\end{proof}


\section{The Problem}

We are interested in solving the multiobjective optimization problem (MOP):
\begin{eqnarray}
\textrm{min}\lbrace F(x): x \in \mathbb{R}^n\rbrace
\label{pom3}
\end{eqnarray}
where $F=\left( F_1, F_2,..., F_m\right): \mathbb{R}^n\longrightarrow \mathbb{R}^m\cup \{+ \infty \}^m$ is a vector function satisfying the following assumption:
\begin{description}
\item [$\bf (C_{1.1})$]  $F$ is a proper lower semicontinuous vector function on $\mathbb{R}^n$, i.e, each $F_i:\mathbb{R}^n\longrightarrow \mathbb{R}^m\cup \{+ \infty \}$, $i=1,...,m$, is  a proper lower semicontinuous function.
\item [$\bf (C_{1.2})$]   $0 \preceq F.$
\end{description}
\subsection{A quasiconvex model in demand theory}
\noindent

Let $n$ be a finite number of consumer goods. A consumer is an agent who must choose how much to consume of each good.  An ordered set of numbers representing the amounts consumed of each good set is called vector of consumption, and denoted by $ x =  (x_1, x_2, ...,  x_n) $ where $ x_i $ with $ i = 1,2,  ..., n $,  is the quantity consumed  of good $i$. Denote by $ X $, the feasible set of these vectors  which will be called the  set of consumption,  usually in economic applications we have $ X  \subset \mathbb{R} ^ n_ + $.

In the classical approach of demand theory, the analysis of consumer behavior starts specifying a preference relation over the set $X,$ denoted by $\succeq$. The notation: $  "x \succeq  y "  $ means that "$ x $ is at least as good as  $ y $" or "$  y $ is not preferred to $x$". This preference relation  $ \succeq  $ is assumed rational, i.e,  is complete because the consumer is able to order all possible combinations of goods, and  transitive, because consumer preferences are consistent, which means if the consumer prefers $\bar{x}$ to $\bar{y} $ and $\bar{y}$ to $\bar{z}$,  then he prefers $\bar{x}$ to $\bar{z} $ (see Definition  3.B.1  of Mas-Colell et al. \cite{Colell}).

The quasiconvex model for a convex preference relation $\succeq,$ is ${\ max}\{ \mu(x) :x \in X\},
$ where $\mu$ is the utility function representing the preference, see Papa Quiroz et al. \cite{PapaLenninOliveira} for more detail. Now consider a multiple criteria, that is, consider  $ m $ convex preference relations denoted by $\succeq_i, i=1,2,...,m.$ Suppose that for each preference $\succeq_i,$  there exists an utility  function, $ \mu_i,$ respectively, then the problem of maximizing the consumer preference on  $ X $ is equivalent to solve the quasiconcave multiobjective optimization problem
\begin{eqnarray*}
\textnormal{(P')\ max}\{ (\mu_{1}(x), \mu_{2}(x), ..., \mu_{m}(x)) \in \mathbb{R}^m :x \in X\}.
\end{eqnarray*}
Since there is not a single point which maximize all the functions simultaneously the concept of optimality is established in terms of Pareto optimality or efficiency. Taking F = $ (- \mu_1, - \mu_2, ..., - \mu_m) $, we obtain a minimization problem with quasiconvex multiobjective function, since each  component function is quasiconvex one.

\section{Exact algorithm}
In this section, to solve the problem $(\ref {pom3}),$ we propose a linear scalarization proximal point algorithm with quadratic regularization using the Fréchet subdifferential,  denoted by {\bf SPP} algorithm. \\ \\
 {\bf SPP Algorithm }
\begin{description}
\item [\bf Initialization:] Choose  an arbitrary starting point
\begin{eqnarray}
x^0\in\mathbb{R}^n 
\label{inicio3}
\end{eqnarray}
\item [Main Steps:] Given $x^k$ finding $x^{k+1} $ such that
 \begin{eqnarray}
 0 \in \hat{\partial}\left( \left\langle F(.), z_k\right\rangle  + \dfrac{\alpha_k}{2} \Vert\ .\  - x^k \Vert ^2  + \delta_{\Omega_k}(.) \right) (x^{k+1}) 
 \label{subdiferencial3}
 \end{eqnarray}
where $\hat{\partial}$ is the Fréchet subdifferential, $\Omega_k= \left\{ x\in \mathbb{R}^n: F(x) \preceq F(x^k)\right\}$, $\alpha_k > 0 $,\\ $\left\{z_k\right\} \subset \mathbb{R}^m_+\backslash \left\{0\right\}$ and $\left\|z_k\right\| = 1$.
\item [Stop criterion:] If $x^{k+1}=x^{k} $ or $x^{k+1}$ is a Pareto critical point, then stop.                                           Otherwise to do $k \leftarrow k + 1 $ and return to Main Steps.
\end{description}

\subsection{Existence of the iterates}
\begin{teorema}
 \label{existe0}
 Let $F:\mathbb{R}^n\longrightarrow \mathbb{R}^m\cup \{+ \infty \}^m$ be a vector function satisfying $\bf (C_{1.1}),$ and $\bf (C_{1.2}),$ then the sequence $\left\{x^k\right\}$, generated by the ${\bf SPP}$ algorithm, is well defined.
\end{teorema}
\begin{proof}
Let $x^0 \in \mathbb{R}^n $ be an arbitrary point given in the initialization step. Given $x^k$, define $\varphi_k(x)=\left\langle F(x), z_k\right\rangle + \frac{\alpha_k}{2}\left\|x - x^k\right\|^2 +\delta_{\Omega_k}(x)$, where $\delta_{\Omega_k}(.)$ is the indicator function of ${\Omega_k}$. Then we have that min$\{\varphi_k(x): x \in \mathbb{R}^n\}$ is equivalent to min$\{\left\langle F(x), z_k\right\rangle + \frac{\alpha_k}{2}\left\|x - x^k\right\|^2: x \in \Omega_k\}$. As $\varphi_k$ is lower semicontinuous and coercive then, using Proposition \ref{coercivaesemicont},  we obtain that there exists $x^{k+1} \in \R^n$ which is a global minimum of $\varphi_k.$ From Proposition \ref{otimo}, $x^{k+1}$ satisfies:
 $$ 0 \in \hat{ \partial}\left( \left\langle F(.), z_k\right\rangle  + \dfrac{\alpha_k}{2} \Vert\ .\  - x^k \Vert ^2 + \delta_{\Omega_k}(.)\right) (x^{k+1})$$
\end{proof}
\subsection{Fejér convergence Property} 
To obtain some desirable properties it is necessary to assume the following assumptions on the function $F$ and the initial point $x^0$ :
\begin{description}
\item [$\bf (C_2)$] $F$ is $\mathbb{R}^m_+$-quasiconvex;

\item [${\bf (C_3)}$] The set $\left(F(x^0) - \mathbb{R}^m_+\right)\cap F(\mathbb{R}^n)$ is $\mathbb{R}^m_+$ - complete,  meaning that for all sequences $\left\{a_k\right\}\subset\mathbb{R}^n$, with $a_0 = x^0$, such that $F(a_{k+1}) \preceq F(a_k)$, there exists  $ a \in \mathbb{R}^n$ such that $F(a)\preceq F(a_k), \ \forall \ k \in \mathbb{N}$.
\end{description}
\begin{obs}
The assumption $ {\bf (C_3)}$ is cited in several works involving the proximal point method for convex functions, see Bonnel et al. \cite{Iusem}, Ceng and Yao \cite {Ceng} and, Villacorta and Oliveira \cite {Villacorta}.
\end{obs}

 \begin{proposicao}
Let $F:\mathbb{R}^n\longrightarrow \mathbb{R}^m\cup \{+ \infty \}^m$ be a function that satisfies the assumptions $\bf (C_{1.1})$ and $\bf (C_2)$.  If $g  \in \hat{\partial}\left( \left\langle F(.), z\right\rangle + \delta_{\Omega} \right)(x)$, with $z \in \mathbb{R}^m_+\backslash \left\{0\right\}$, and $F(y) \preceq F(x)$, with $y \in \Omega$, and $\Omega \subset \mathbb{R}^n$ a closed and convex set, then 
$\left\langle g , y - x\right\rangle \leq 0$.
\label{propfejer2}
\end{proposicao}
\begin{proof}
Let $t \in \left( 0, 1\right]$, then from the $\mathbb{R}^m_+$-quasiconvexity of $F$ and the assumption that $F(y) \preceq F(x)$, we have: $F_i(ty + (1-t)x) \leq \textrm{max}\left\{F_i(x), F_i(y)\right\} = F_i(x),\ \forall \ i \in \lbrace 1,...m\rbrace$.  It follows that for each $z \in \mathbb{R}^m_+\backslash \left\{0\right\}$, we have
\begin{equation}
 \left\langle F(ty + (1-t)x) , z\right\rangle \leq \left\langle F(x) , z\right\rangle.
 \label{Fz2}
 \end{equation}
 As $g  \in \hat{\partial}\left( \left\langle F(.), z\right\rangle + \delta_{\Omega} \right)(x)$, we obtain 
\begin{equation}
\left\langle F(ty + (1-t)x) , z\right\rangle + \delta_{\Omega}(ty + (1-t)x)\geq \left\langle F(x) , z\right\rangle +  \delta_{\Omega}(x) + t\left\langle g, y - x\right\rangle + o(t\left\|y - x\right\|)
\label{soma2}
\end{equation}
From $(\ref{Fz2})$ and $(\ref{soma2})$, we conclude
\begin{equation}
t\left\langle g, y - x\right\rangle + o(t\left\|y - x\right\|) \leq 0
\label{somafim2}
\end{equation} 
\\
On the other hand, we have
 $\lim\limits_{t \rightarrow 0}\frac{o(t\left\|y - x\right\|)}{t\left\|y - x\right\|}= 0$.  Thus,
$\lim\limits_{t\rightarrow 0}\frac{o(t\left\|y - x\right\|)}{t}=\lim\limits_{t\rightarrow 0}\frac{o(t\left\|y - x\right\|)}{t\left\|y - x\right\|}\left\|y - x\right\|=0$.
Therefore, dividing $(\ref{somafim2})$ by $t$ and taking $t \rightarrow 0$, we obtain the desired result.
\end{proof}

Observe that if the sequence $\left\{x^k\right\}$ generated by the {\bf SPP} algorithm satisfies the assumption ${\bf (C_3)}$ then the set 
\begin{center}
$E = \left\{x \in \mathbb{R}^n: F\left(x\right)\preceq F\left(x^k\right),\ \ \forall\  k \in \mathbb{N}\right\}$
\end{center}
 is nonempty.
 
\begin{proposicao}
Under assumptions ${\bf(C_{1.1})}$, ${\bf(C_{1.2})},$ ${\bf(C_2)}$ and ${\bf(C_3)}$ the sequence $\left\{x^k\right\}$, generated by the {\bf  SPP} algorithm, $(\ref{inicio3})$ and $(\ref{subdiferencial3}),$ is Fejér convergent to $E$.
\label{fejer0}
\end{proposicao}

\begin{proof} Observe that $\forall \ x \in \mathbb{R}^n$:

\begin{eqnarray}
 \left\|x^k - x\right\|^2  =  \left\|x^k - x^{k+1}\right\|^2 + \left\|x^{k+1} - x\right\|^2 +
                              2\left\langle x^k - x^{k+1}, x^{k+1} - x\right\rangle.
\label{norma2}
\end{eqnarray}
From Theorem \ref{existe0}, $(\ref{subdiferencial3})$ and from Proposition \ref{somafinita}, $(iv)$, we have that there exists $g_k \in \hat{\partial}\left( \left\langle F(.), z_k\right\rangle + \delta_{\Omega_k}\right)(x^{k+1})$ such that:
\begin{equation}
 x^k - x^{k+1} = \dfrac{1}{\alpha_k}g_k
 \label{xk}
 \end{equation}
 Now take $x^* \in E$, then $x^* \in \Omega_k$ for all $k \in \mathbb{N}$. Combining $(\ref{norma2})$ with $x = x^*$ and $(\ref{xk})$, we obtain:
{\footnotesize
\begin{eqnarray}
\left\|x^k - x^*\right\|^2   =  \left\|x^k - x^{k+1}\right\|^2 + \left\|x^{k+1} - x^* \right\|^2 + 
                               \frac{2}{\alpha_k}\left\langle g_k ,\  x^{k+1} -x^*\right\rangle 
                             \geq  \left\|x^k - x^{k+1}\right\|^2 + \left\|x^{k+1} - x^* \right\|^2\nonumber \\
\label{desigualdade0}
\end{eqnarray}
}
where the last  inequality follows from Proposition \ref{propfejer2}. From $(\ref{desigualdade0})$, it implies that
\begin{eqnarray}
 0\leq \left\|x^{k+1} - x^k\right\|^2 \leq \left\|x^k - x^*\right\|^2 - \left\|x^{k+1} - x^*\right\|^2.
 \label{desigual0}
 \end{eqnarray}
 Thus,
\begin{equation}
\left\|x^{k+1} - x^*\right\| \leq \left\|x^k - x^*\right\|
\label{fejer70}
\end{equation}
\end{proof}

\begin{proposicao}
Under  assumptions ${\bf(C_{1.1})}, {\bf(C_{1.2})},$ ${\bf(C_2)}$ and ${\bf(C_3)},$ the sequence $\left\{x^k\right\}$ generated by the {\bf SPP} algorithm, $(\ref{inicio3})$ and $(\ref{subdiferencial3}),$ satisfies
\begin{center}
$\lim  \limits_{k\rightarrow +\infty}\left\|x^{k+1} - x^k\right\| = 0$.
\label{decrescente00}
\end{center}
\end{proposicao}

\begin{proof}
It follows from $(\ref{fejer70})$ that $ \forall\  x^* \in E$,  $\left\{\left\|x^k - x^*\right\|\right\}$ is a nonnegative and nonincreasing sequence, and hence is convergent.  Thus, the right-hand side of $(\ref{desigual0})$ converges to 0 when $k \rightarrow +\infty$, and the result is obtained.
\end{proof}


\subsection{Convergence Analysis I: non differentiable case }
In this subsection we analyze the convergence of the proposed algorithm when $F$ is a non differentiable vector function.

\begin{proposicao}
Under assumptions ${\bf(C_{1.1})}, {\bf(C_{1.2})},$ ${\bf(C_2)}$ and ${\bf(C_3)},$ the sequence  $\left\{x^k\right\}$  generated by the {\bf SPP} algorithm converges to some point of $E$.
\label{acumulacao10}
\end{proposicao}

\begin{proof}
From Proposition $\ref{fejer0}$ and Lemma $\ref{fejerlim1}$, $(i)$, $\left\{x^k\right\}$ is bounded, then there exists a subsequence $\left\{x^{k_j}\right\}$ such that $\lim  \limits_{j\rightarrow +\infty}x^{k_j} = \widehat{x}$.  Since $\left\langle F(.), z\right\rangle$ is lower semicontinuos function for all  $z \in \mathbb{R}^m_+ \backslash \left\{0\right\}$ then
$\left\langle F(\widehat x),z\right\rangle \leq \liminf \limits_{j\rightarrow +\infty}\left\langle F(x^{k_j}) , z\right\rangle $.  On the other hand, $x^{k+1} \in \Omega_k$ so $\left\langle F(x^{k+1}) , z\right\rangle \leq \left\langle F(x^{k}) , z\right\rangle $. Furthermore, from assumption ${\bf(C_{1.2})}$ the function $\left\langle F(.), z\right\rangle$ is  bounded below for each $z \in \mathbb{R}^m_+\backslash \left\{0\right\},$ then, the sequence $\left\{\left\langle F(x^k),z\right\rangle\right\}$ is nonincreasing and bounded below, hence convergent. Therefore 
{\small
\begin{center}
$\left\langle F(\widehat {x}),z\right\rangle \leq \liminf \limits_{j\rightarrow +\infty}\left\langle F(x^{k_j}) , z\right\rangle = \lim \limits_{j\rightarrow +\infty}\left\langle F(x^{k_j}) , z\right\rangle =  inf_{k\in \mathbb{N}}\left\{\left\langle F(x^k),z\right\rangle\right\}\leq \left\langle F(x^k),z\right\rangle$.
\end{center}
} 
It follows that
$\left\langle F(x^k)-F(\widehat{x}),z\right\rangle \geq  0, \forall \ k \in \mathbb{N}, \forall \ z \in \mathbb{R}^m_+ \backslash \left\{0\right\}$. We conclude that $F(x^k) - F(\widehat{x}) \in \mathbb{R}^m_+$, i.e, $F(\widehat{x})\preceq F(x^k), \forall \ k  \in \mathbb{N}$. Therefore $\widehat{x}\in E,$ and by Lemma $\ref{fejerlim1}$, $(ii)$, we get the result.
\end{proof}
\subsubsection{Convergence  to a weak Pareto solution}
\begin{teorema}
Let $F:\mathbb{R}^n\longrightarrow \mathbb{R}^m $ be a continuous vector function satisfying the assumptions $\bf (C_{1.2}),$ $\bf (C_2)$ and $\bf (C_3)$.  If $ \lim \limits_{k\rightarrow +\infty}\alpha_k= 0$ and the iterations are given in the form
\begin{eqnarray}
  x^{k+1}\in \textnormal{arg min} \left\{\left\langle F(x), z_k\right\rangle+\frac{\alpha_k}{2}\left\|x - x^k\right\|^2 : x\in\Omega_k\right\},
  \label{recursao0}
 \end{eqnarray}
then the sequence $\lbrace x^k\rbrace$ converges to a weak Pareto solution of the problem $(\ref{pom3})$.
\end{teorema}
\begin{proof}
Let\ \  $x^{k+1}\in \textnormal{arg min} \left\{\left\langle F(x), z_k\right\rangle+\frac{\alpha_k}{2}\left\|x - x^k\right\|^2 : x\in\Omega_k\right\}$, this implies that 
{\small
\begin{eqnarray}
    \left\langle F(x^{k+1}), z_k\right\rangle +\frac{\alpha_k}{2}\left\|x^{k+1}-x^k\right\|^2 \leq \left\langle F(x), z_k\right\rangle + 
     \frac{\alpha_k}{2}\left\|x -x^k\right\|^2,
    \label{des0}
 \end{eqnarray}
 }
$\forall \ x \in \Omega_k$. Since the sequence $\left\{x^k\right\}$ converges to some point of $E$, then exists $x^* \in E$ such that $\lim \limits_{k\rightarrow +\infty}x^{k}= x^*$. Since that $\left\{z_k\right\}$ is bounded, there exists a subsequence $\left\{z_{k_l}\right\}_{l\in \mathbb{N}}$ such that $\lim \limits_{l\rightarrow +\infty}z_{k_l}=\bar{z}$, with $\bar{z} \in \mathbb{R}^m_+\backslash \left\{0\right\}$.  Taking $k=k_l$ in $(\ref{des0})$, we have
{\small
\begin{eqnarray}
    \left\langle F(x^{k_l+ 1}), z_{k_l}\right\rangle +\frac{\alpha_{k_l}}{2}\left\|x^{k_l+1}-x^{k_l}\right\|^2 \leq \left\langle F(x), z_{k_l}\right\rangle + \frac{\alpha_{k_l}}{2}\left\|x - x^{k_l}\right\|^2. 
 \label{desi0}
 \end{eqnarray}
 }
$\forall \ x \in E.$ As
\begin{center}
$ \frac{\alpha_{k_l}}{2}\left\|x^{k_{l+1}}-x^{k_l}\right\|^2 \rightarrow 0$ and $\frac{\alpha_{k_l}}{2}\left\|x - x^{k_l}\right\|^2 \rightarrow0$ when $l \rightarrow +\infty$
\end{center}
and from the continuity of $F$, taking $l \rightarrow +\infty$ in $(\ref{desi0})$, we obtain
 \begin{eqnarray}
 \left\langle F(x^*), \overline{z}\right\rangle \leq \left\langle F(x), \overline{z}\right\rangle, \forall \  x \in E
 \label{minfi0}
 \end{eqnarray}
Thus $x^* \in \textrm{arg min} \left\{\left\langle F(x), \overline{z}\right\rangle: x \in E\right\}$.
Now, $\left\langle F(.), \overline{z}\right\rangle$, with $\bar{z} \in \mathbb{R}^m_+\backslash \left\{0\right\}$ is a strict scalar representation of $F$, so a weak scalar representation, then by Proposition $\ref{inclusao}$ we have that $x^* \in \textrm{arg min}_w \left\{F(x):x \in E \right\}$.\\ 
We shall prove that $x^* \in \textrm{arg min}_w \left\{F(x):x \in \mathbb{R}^n \right\}$. Suppose by contradiction that $x^* \notin \textrm{arg min}_w \left\{F(x):x \in \mathbb{R}^n \right\}$ then there exists $\widetilde{x} \in \mathbb{R}^n$ such that
  \begin{equation}
   F(\widetilde{x})\prec F(x^*)
    \label{pareto0}
   \end{equation}
So for $\bar{z} \in \mathbb{R}^m_+\backslash \left\{0\right\}$ it follows that
  \begin{equation}
  \left\langle F(\widetilde{x}), \bar{z}\right\rangle < \left\langle F(x^*), \bar{z}\right\rangle 
  \label{d0}
  \end{equation}
Since $x^* \in E$, from $(\ref{pareto0})$ we conclude that $\widetilde{x} \in E$.  Therefore from $(\ref{minfi0})$ and $(\ref{d0})$ we obtain a contradiction.
 \end{proof}\\ 
 
\subsubsection{Convergence to a generalized critical point}
\begin{teorema}
Let $F:\mathbb{R}^n\longrightarrow \mathbb{R}^m $ be a continuous vector function satisfying the assumptions $\bf (C_{1.2}),$ $\bf (C_2)$ and $\bf (C_3)$.  If $0 < \alpha_k < \tilde{\alpha}$ then the sequence $\lbrace x^k\rbrace$  generated by the {\bf SPP} algorithm, $(\ref{inicio3})$ and $(\ref{subdiferencial3})$ satisfies 
$$\lim  \limits_{k\rightarrow +\infty}g^{k} = 0,$$
where $g^k \in \hat{\partial}\left( \left\langle F(.), z_k\right\rangle + \delta_{\Omega_k}\right)(x^{k+1})$.
\label{geral}
\end{teorema}
\begin{proof}
From Theorem $\ref{existe0}$, $(\ref{subdiferencial3})$ and from Proposition $\ref{somafinita}$, $(iv)$, there exists a vector $g_k \in \hat{\partial}\left( \left\langle F(.), z_k\right\rangle + \delta_{\Omega_k}\right)(x^{k+1})$ such that $g^k = \alpha_k(x^k - x^{k+1})$.  Since $0 < \alpha_k < \tilde{\alpha}$ then 
\begin{equation}
0 \leq \Vert g^k\Vert \leq \tilde{\alpha}\left\|x^k - x^{k+1} \right\|
\label{beta0}
\end{equation} 
From Proposition $\ref{decrescente00}$, $\lim  \limits_{k\rightarrow +\infty}\left\|x^{k+1} - x^k\right\| = 0$, and from $(\ref{beta0})$ we have $\lim  \limits_{k\rightarrow +\infty}g^k = 0$.
\end{proof}
\subsubsection{Finite Convergence to a Pareto Optimal Point}
\noindent
Following the paper of Bento et al, \cite{Bento} subsection 4.3, it is possible to prove the convergence of a special particular case of the proposed algorithm to a Pareto optimal point of the problem (\ref{pom3}). Let $F:\mathbb{R}^n\longrightarrow \mathbb{R}^m\cup \{+ \infty \}^m$ be a proper lower semicontinuous convex function and consider the following particular iteration of (\ref{subdiferencial3}):
\begin{eqnarray}
  x^{k+1}= \textnormal{arg min} \left\{\left\langle F(x), z\right\rangle+\frac{\alpha_k}{2}\left\|x - x^k\right\|^2 : x\in\R^n\right\},
  \label{recursao0f}
\end{eqnarray}
where $z\in \mathbb{R}^m_+\backslash \left\{0\right\}$ such that $||z||=1.$
\begin{Def}
Consider the set of Pareto optimal points of $(\ref{pom3})$, denoted by $Min(F)$ and let $\bar x\in Min(F)$. We say that $Min(F)$ is $W_{F(\bar x)}$-weak sharp minimum for the problem $(\ref{pom3})$ if there exists a constant $\tau>0$ such that
$$
F(x)-F(\bar x)\notin B(0,\tau d(x,W_{F(\bar x)}))-\R^m_{+},\;\; x\in \R^n\backslash \,W_{F(\bar x)},
$$ 
where $d(x,Z)=\inf \{d(x,z):z\in Z\}$ and $W_{p}=\{x\in \R^n: F(x)=F(p)\}.$
\end{Def}
\begin{teorema}
Let $F$ be a proper lower semicontinuous convex vector function satisfying the assumptions $\bf (C_{1.2}),$ and $\bf (C_3).$ Assume that $\{x^k\}$ is a sequence generated from the {\bf SPP} algorithm with $x^{k+1}$ being generates from $(\ref{recursao0f})$. Consider also that the set of Pareto optimal points of $(\ref{pom3})$ is nonempty and assume that $Min(F)$ is $W_{F(\bar x)}$-weak sharp minimum for the problem $(\ref{pom3})$ with constant $\tau>0$ for some $\bar x\in Min(F).$ Then the sequence $\{x^k\}$ converges, in a finite number of iterations, to a Pareto optimal point.
\end{teorema}
\begin{proof}
Simmilar to the proof of Theorem 4.3 of Bento et al., \cite{Bento}.
\end{proof}

\subsection{Convergence analysis II: Differentiable Case}
In this subsection we analyze the convergence of the method when $F$ satisfies the following assumption:

\begin{description}
\item [$\bf (C_4)$]  $F:\R^n\longrightarrow \R^m$ is a continuously differentiable vector function on $\mathbb{R}^n$.
\end{description}
 The next proposition characterizes a quasiconvex differentiable vector functions.
\begin{proposicao}
Let $F :\mathbb{R}^n\longrightarrow \mathbb{R}^m $ be a differentiable function satisfying the assumption ${\bf(C_2)}$, ${\bf(C_3)}.$ If $x \in E,$  then $\left\langle \nabla F_i(x^k) , x - x^k\right\rangle \leq 0$, $\forall \ k \in \mathbb{N}$ and $\forall \  i \in \left\{1,...,m\right\}$.
\label{caracterizacaodif}
\end{proposicao}
\begin{proof}
Since \ $F$ is $\mathbb{R}^m_+$-quasiconvex each $F_i,$ $ i = 1,..., m,$ is quasiconvex.                                                     Then the result follows from the classical characterization of the scalar differentiable quasiconvex functions, see
 $\textrm{see Mangasarian \cite{Mangasarian}, p.134}$.
\end{proof}


\begin{teorema}
\label{conv3}
Let $F:\mathbb{R}^n\longrightarrow \mathbb{R}^m $ be a function satisfying the assumptions $\bf (C_2)$, $\bf (C_3)$ and $\bf (C_4)$. If $0 < \alpha_k < \tilde{\alpha}$, then the sequence $\lbrace x^k\rbrace$  generated by the {\bf SPP} algorithm, $(\ref{inicio3})$ and $(\ref{subdiferencial3}),$ converges to a Pareto critical point of the problem $(\ref{pom3})$.
\label{teoparetocri}
\end{teorema}
\begin{proof}
 In Proposition \ref{acumulacao10} we prove that there exists $\widehat{x} \in E$ such that $\lim \limits_{k\rightarrow +\infty}x^{k}= \widehat{x}$.  From Theorem $\ref{existe0}$ and  $(\ref{subdiferencial3})$, we have
 \begin{equation*}
 0 \in \hat{\partial}\left( \left\langle F(.), z_k\right\rangle  + \dfrac{\alpha_k}{2}\Vert\ .\  - x^k \Vert ^2  + \delta_{\Omega_k}(.) \right) (x^{k+1})
 \end{equation*}
 Due to Proposition $\ref{somafinita}$, $(iv)$, we have 
\begin{center}
$0 \in \nabla\left(\left\langle F(.) , z_k\right\rangle\right)(x^{k+1})+ \alpha_k \left(x^{k+1} - x^k \right) + \mathcal{N}_{\Omega_k}(x^{k+1})$
\end{center} 
where $\mathcal{N}_{\Omega_k}(x^{k+1})$ is the normal cone to $\Omega_k$ at $x^{k+1}\in \Omega_k$.\\ 
So there exists $\nu_k \in \mathcal{N}_{\Omega_k}(x^{k+1})$ such that:  
\begin{equation}
0 = \sum_{i=1}^m \nabla F_i(x^{k+1})(z_k)_i + \alpha_k\left(x^{k+1} - x^k \right) + \nu_k. 
\label{otimalidade0}
\end{equation}
 Since $\nu_k \in \mathcal{N}_{\Omega_k}(x^{k+1})$ then 
\begin{equation}
\left\langle \nu_k \ ,\ x - x^{k+1}\right\rangle \leq\  0,\  \forall \ x \in \Omega_k.
\label{cone0}
\end{equation}
Take $\bar{x} \in E$.  By definition of $E$, $\bar{x} \in  \Omega_k$ for all $k \in \mathbb{N}$.  Combining $(\ref{cone0})$ with $x = \bar{x}$ and $(\ref{otimalidade0})$, we have
{\small
\begin{eqnarray}
\left\langle \sum_{i=1}^m \nabla F_i(x^{k+1})(z_k)_i , \bar{x} - x^{k+1}\right\rangle + \alpha_k\left\langle x^{k+1} - x^k , \bar{x} - x^{k+1}\right\rangle\geq 0.&
\label{cone00}
\end{eqnarray}
}
Since that $\left\{z_k\right\}$ is bounded, then there exists a subsequence $ \left\{z_{k_j}\right\}_{j \in \mathbb{N}}$ such that $\lim \limits_{j\rightarrow +\infty}z_{k_j}= \bar{z}$ with $\bar{z} \in \mathbb{R}^m_+\backslash \left\{0\right\}$.  Thus the inequality in $(\ref{cone00})$ becomes
{\small
\begin{eqnarray}
\left\langle \sum_{i=1}^m \nabla F_i(x^{{k_j}+1})(z_{k_j})_i , \bar{x} - x^{{k_j}+1}\right\rangle + \alpha_{k_j}\left\langle x^{{k_j}+1} - x^{k_j} , \bar{x} - x^{{k_j}+1}\right\rangle\geq 0.&
\label{cone000}
\end{eqnarray}
}
Since $\left\{x^k\right\}$ and $\left\{ \alpha_k\right\}$ are bounded, $\lim  \limits_{k\rightarrow +\infty}\left\|x^{k+1} - x^k\right\| = 0$ and $F$ is continuously differentiable, the inequality in $(\ref{cone000})$, for all $\bar{x}\in E$, becomes:
{\small
\begin{eqnarray}
\left\langle\sum_{i=1}^m \nabla F_i(\widehat{x})\bar{z}_i \ ,\ \bar{x} -  \widehat{x}\right\rangle \geq 0
\Rightarrow  \sum_{i=1}^m \bar{z}_i\left\langle \nabla F_i(\widehat{x}) \ ,\ \bar{x} -  \widehat{x} \right\rangle\geq 0.
\label{somai0}
\end{eqnarray}
}
From the quasiconvexity of each component function $F_i$, for each $i \in \left\{1,...,m\right\}$, we have that\\ 
$\left\langle \nabla F_i(\widehat{x})\ ,\ \bar{x} -  \widehat{x} \right\rangle\leq 0$ and because $\bar{z} \in \mathbb{R}^m_+\backslash \left\{0\right\}$, from $(\ref{somai0})$, we obtain
\begin{eqnarray}
\sum_{i=1}^m \bar{z}_i\left\langle \nabla F_i(\widehat{x}) \ ,\ \bar{x} -  \widehat{x} \right\rangle = 0.
\label{somai00}
\end{eqnarray}
Without loss of generality consider the set $J = \left\{i \in I: \bar{z}_i > 0 \right\}$, where $I = \left\{1,...,m\right\}$.  Thus, from $(\ref{somai00})$, for all $\bar{x} \in E$ we have
\begin{eqnarray}
\left\langle \nabla F_{i}(\widehat{x}) \ ,\ \bar{x} -  \widehat{x} \right\rangle = 0,\  \forall \ \ i \in J.
\label{ecritico0}
\end{eqnarray}
Now we will show that $\widehat{x}$ is a Pareto critical point.  \\
Suppose  by contradiction that $\widehat{x}$ is not a Pareto critical point, then there exists a direction $v \in \mathbb{R}^n$ such that $JF(\widehat{x})v \in  -\mathbb{R}^m_{++}$, i.e, 
\begin{eqnarray}
\left\langle \nabla F_i(\widehat{x}) , v \right\rangle < 0, \forall \ i \in \left\{1,...,m\right\}.
\label{direcao10}
\end{eqnarray}
Therefore $v$ is a descent direction for the multiobjective function $F$ in $\widehat{x}$, so, $\exists \ \varepsilon > 0$ such that
\begin{eqnarray}
  F(\widehat{x} + \lambda v) \prec F(\widehat{x}),\ \forall \ \lambda \in (0, \varepsilon]. 
\label{descida10}
\end{eqnarray}
Since \ $\widehat{x}\ \in\  E$, then from $(\ref{descida10})$ we conclude that $\widehat{x} + \lambda v \in E$.  Thus, from $(\ref{ecritico0})$ with $\bar{x} = \widehat{x} + \lambda v $, we obtain: {\small $\left\langle \nabla F_{i}(\widehat{x})\ ,\  \widehat{x} + \lambda v - \widehat{x} \right\rangle = \left\langle \nabla F_{i}(\widehat{x})\ ,\ \lambda v  \right\rangle = \lambda\left\langle \nabla F_{i}(\widehat{x})\ ,\ v \right\rangle = 0$}.\\
It follows that $\left\langle \nabla F_{i}(\widehat{x})\ ,\ v \right\rangle = 0$ for all $ i \in J,$ contradicting $(\ref{direcao10})$.  Therefore $\widehat{x}$ is Pareto critical point of the problem $(\ref{pom3})$.
\end{proof}
 \section{An inexact proximal algorithm}
In this section we present an inexact version of the {\bf SPP} algorithm, which we denote by {\bf ISPP} algorithm.

\subsection{{\bf ISPP} Algorithm}

Let $F: \mathbb{R}^n \rightarrow \mathbb{R}^m$ be a vector function satisfying the assumptions ${(\bf C_2)}$ and ${(\bf C_4)}$, and consider two sequences: the proximal parameters $\left\{\alpha_k\right\}$ and the sequence $ \left\{z_k \right\}\subset \mathbb{R}^m_+\backslash \left\{0\right\}$ with $\left\|z_k\right\| = 1$.
\begin{description}
\item [\bf Initialization:] Choose  an arbitrary starting point
\begin{eqnarray}
x^0\in\mathbb{R}^n 
\label{inicio2}
\end{eqnarray}
\item [Main Steps:] Given $x^k,$ define the function  $\Psi_k : \mathbb{R}^n \rightarrow \mathbb{R} $ such that $ \Psi_k (x) = \left\langle F(x), z_k\right\rangle $ and consider $\Omega_k = \left\{ x\in \mathbb{R}^n: F(x) \preceq F(x^k)\right\}$.  Find $x^{k+1}$ satisfying
 \begin{equation}
 0 \in \hat{\partial}_{\epsilon_k}\Psi_k (x^{k+1}) + \alpha_k\left(x^{k+1} - x^k\right) + \mathcal{N}_{\Omega_k}(x^{k+1}),
 \label{diferencial}
 \end{equation}
 \begin{equation}
 \label{delta}
\displaystyle \sum_{k=1}^{\infty} \delta_k < +\infty,
\end{equation}
where $\delta_k = \textnormal{max}\left\lbrace  \dfrac{\varepsilon_k}{\alpha_k}, \dfrac{\Vert \nu_k\Vert}{\alpha_k}\right\rbrace,$ $\varepsilon_k\geq 0,$  and $\hat{\partial}_{\varepsilon_k}$ is the Fréchet $\varepsilon_k$-subdifferential.
\item [Stop criterion:] If $x^{k+1}=x^{k} $ or $x^{k+1}$ is a Pareto critical point, then stop.                                           Otherwise to do $k \leftarrow k + 1 $ and return to Main Steps.
\end{description}
\subsubsection{Existence of the iterates}
\begin{proposicao}
 Let $F:\mathbb{R}^n\longrightarrow \mathbb{R}^m $ be a vector function satisfying the assumptions $\bf (C_{1.2}),$ $\bf (C_2)$ and $\bf (C_4)$.  Then the sequence $\left\{x^k\right\}$ generated by the  {\bf ISPP} algorithm, is well defined.
\label{iteracao2}
\end{proposicao}
\begin{proof}
Consider $x^0 \in \mathbb{R}^n $ given by (\ref{inicio2}).  Given $x^k$, we will show that there exists $x^{k+1}$ satisfying the condition $(\ref{diferencial})$.  Define the function $\varphi_k(x) = \Psi_k (x)+\frac{\alpha_k}{2}\left\|x - x^k\right\|^2  + \delta_{\Omega_k}(x)$.  Analogously  to the proof of Theorem\  $\ref{existe0}$ there exists $x^{k+1} \in \Omega_k$ which is a global minimum of $\varphi_k (.),$ so, from Proposition $\ref{otimo}$, $x^{k+1}$ satisfies
\begin{center}
 $ 0 \in \hat{ \partial}\left( \Psi_k(.)  + \dfrac{\alpha_k}{2}\Vert\ .\  - x^k \Vert ^2 + \delta_{\Omega_k}(.)\right) (x^{k+1})$.
 \end{center}
From Proposition $\ref{somafinita},$ $\ (iii)$ and $\ (iv)$, we obtain 
\begin{center}
$0 \in \hat{\partial}\Psi_k(x^{k+1})+ \alpha_k \left(x^{k+1} - x^k \right) + \mathcal{N}_{\Omega_k}(x^{k+1})$.  
\end{center}
From Remark $\ref{frechet2}$, $x^{k+1}$ satisfies $(\ref{diferencial})$ with $\varepsilon_k = 0$.
\end{proof}
\begin{obs}
\label{bregman}
From the inequality $(a-1/2)^2\geq 0, \forall a\in \R,$ we obtain the following relation
\begin{center}
 $\Vert x - z\Vert^2 + \frac{1}{4} \geq \Vert x - z\Vert, \ \forall x,z \in \mathbb{R}^n$
 \end{center}
\end{obs} 
\begin{proposicao}
\label{fejer3}
Let $\left\{x^k\right\}$ be a sequence generated by the {\bf ISPP} algorithm. If the assumptions ${\bf(C_{1.2})},$ ${\bf(C_2)}$, ${\bf(C_3)}$, ${\bf(C_4)}$ and $(\ref{delta})$  are satisfied, then for each $\hat{x} \in E$, $\{ \left\|\hat{x} - x^k\right\|^ 2 \}$ converges and $\{x^k\}$ is bounded.
\end{proposicao}
\begin{proof}
From $(\ref{diferencial})$, there exist $g_k \in \hat{\partial}_{\varepsilon_k}\Psi_k (x^{k+1})$ and  $\nu_k \in \mathcal{N}_{\Omega_k}(x^{k+1})$ such that
\begin{center}
$0 = g_k + \alpha_k\left(x^{k+1} - x^k \right) + \nu_k$. 
\end{center}
It follows that for any $x \in \mathbb{R}^n$, we obtain
\begin{equation*}
  \langle - g_k , x - x^{k+1}\rangle + \alpha_k\langle x^k - x^{k+1},  x - x^{k+1}\rangle = \langle\nu_k,  x - x^{k+1}\rangle \leq \Vert \nu_k\Vert \Vert  x - x^{k+1}\Vert
\label{diferenca2}
\end{equation*}
Therefore
\begin{equation}
\langle x^k - x^{k+1},  x - x^{k+1} \rangle \leq \dfrac{1}{\alpha_k}\left( \langle g_k , x - x^{k+1}\rangle + \Vert \nu_k\Vert \Vert  x - x^{k+1}\Vert\right) .
\label{a}
\end{equation}
Note that $\forall \ x \in \mathbb{R}^n$:
\begin{eqnarray}
\left\|x - x^{k+1}\right\|^2 &-&\left\| x - x^k \right\|^2  \leq \ \ 2\left\langle x^k - x^{k+1}, x - x^{k+1} \right\rangle .
 \label{norma5}                        
\end{eqnarray}
From $(\ref{a})$ and $(\ref{norma5})$, we obtain
\begin{equation}
\left\|x - x^{k+1}\right\|^2 -\left\| x - x^k \right\|^2  \leq \dfrac{2}{\alpha_k}\left( \langle g_k , x - x^{k+1}\rangle + \Vert \nu_k\Vert \Vert  x - x^{k+1}\Vert\right) .
\label{b}
 \end{equation}
On the other hand, let $\Psi_k (x) = \left\langle F(x), z_k\right\rangle$, where $F: \mathbb{R}^n \rightarrow \mathbb{R}^m$ is continuously differentiable vector function, then $ \Psi_k : \mathbb{R}^n \rightarrow \mathbb{R}$ is continuously differentiable with gradient denoted by $\nabla \Psi_k $.  From Proposition $\ref{fdif}$, we have
\begin{equation}
\partial_{\varepsilon_{k}}\Psi_k (x) = \nabla \Psi_k (x) + \varepsilon_{k} B,
\end{equation}
where $B$ is the closed unit ball in $\R^n$ centered at zero. Futhermore, $\hat{\partial}_{\varepsilon_k} \Psi_k(x) \subset \partial_{\varepsilon_k} \Psi_k(x)$, (see (2.12) in Jofré et al. \cite{Jofre}).  As $g_k \in \hat{\partial}_{\epsilon_k}\Psi_k(x^{k+1})$, we have that $g_k \in \partial_{\epsilon_k}\Psi_k(x^{k+1})$, then $$g_k = \nabla\Psi_k(x^{k+1}) + \varepsilon_k h_k,$$ with $\Vert h_k \Vert \leq 1 $. Now take $\hat{x}\in \textnormal{E},$ then
\begin{eqnarray}
\langle g_k,\hat{x} - x^{k+1} \rangle & =& \left\langle \nabla \Psi_k(x^{k+1}) + \varepsilon_k h_k\ , \ \hat{x} - x^{k+1}\right\rangle \nonumber \\
&=&\sum_{i=1}^m \left\langle \nabla F_i(x^{k+1})\ ,\ \hat{x} -  x^{k+1}\right\rangle (z_k)_i  +\varepsilon_k\left\langle h_k\ ,\ \hat{x} -x^{k+1} \right\rangle \nonumber\\
\label{fe1}
\end{eqnarray}
From Proposition $\ref{caracterizacaodif}$, we conclude that $(\ref{fe1})$ becomes
\begin{eqnarray}
\langle g_k,\hat{x} - x^{k+1} \rangle \leq \varepsilon_k\left\langle h_k\ ,\ \hat{x} -x^{k+1} \right\rangle \leq \varepsilon_k \Vert\hat{x} - x^{k+1} \Vert \nonumber\\
\label{f2}
\end{eqnarray}
From Remark $\ref{bregman}$ with $x = \hat{x}$ and $z = x^{k+1}$, follows
\begin{eqnarray}  
\Vert \hat{x} - x^{k+1}\Vert \leq \left( \Vert \hat{x} - x^{k+1}\Vert ^2 + \frac{1}{4}\right).
\label{c}  
\end{eqnarray}
Consider $x = \hat{x}$ in $(\ref{b})$, using $(\ref{f2})$, $(\ref{c})$ and the condition (\ref{delta}) we obtain 
\begin{eqnarray*}
\left\|\hat{x} - x^{k+1}\right\|^2 -\left\| \hat{x} - x^k \right\|^2 & \leq & \dfrac{2}{\alpha_k}\left( \varepsilon_k + \Vert \nu_k\Vert\right)\Vert \hat{x} - x^{k+1}\Vert  \\
&\leq &  4 \delta_k \left\|\hat{x} - x^{k+1}\right\|^2 + \delta_k.
\end{eqnarray*}
Thus
\begin{eqnarray}
\left\|\hat{x} - x^{k+1}\right\|^2 \leq \left(\frac{1}{1-4\delta_k}\right)\left\|\hat{x} - x^{k}\right\|^2+ \frac{\delta_k}{1-4\delta_k}.
\label{d}
\end{eqnarray}
The condition (\ref{delta}) guarantees that
\begin{center}
$\delta_k < \dfrac{1}{4}, \ \ \forall k > k_0, $\\
\end{center}
where $k_0$ is a natural number sufficiently large, and so,
\begin{equation*}
1 \leq \frac{1}{1-4\delta_k} \leq 1+2\delta_k <2,\ \ \ \textnormal{for}\ \  k \geq k_0, 
\end{equation*}
combining with $(\ref{d})$, results in
\begin{equation}
\left\|\hat{x} - x^{k+1}\right\|^2 \leq \left(1 + 2\delta_k \right)\left\|\hat{x} - x^{k}\right\|^2 + 2\delta_k .
\label{e2}
\end{equation}
Since $\displaystyle \sum_{i=1}^{\infty} \delta_k < \infty,$ applying Lemma $\ref{p}$ in the inequality $(\ref{e2})$, we obtain the convergence of $\{ \|\hat{x} - x^k\|^2 \}$, for each $\hat{x} \in E,$ which implies that there exists $M \in \mathbb{R}_+$, such that $\left\|\hat{x} - x^k \right\|  \leq M, \ \ \forall \ k \in \mathbb{N}.
$
Now, since that $\Vert x ^k\Vert \leq \Vert x ^k -  \hat{x} \Vert + \Vert \hat{x} \Vert,$ we conclude that $\{x^k\}$ is bounded, and so, we guarantee that the set of accumulation points of this sequence is nonempty. \end{proof}

\subsubsection{Convergence of the {\bf ISPP} algorithm}

\begin{proposicao}{(\bf Convergence to some point of E)}\\
If the assumptions ${\bf(C_{1.2})},$ ${\bf(C_2)}$, ${\bf(C_3)}$ and ${\bf(C_4)}$ are satisfied, then the sequence  $\left\{x^k\right\}$  generated by the  {\bf ISPP} algorithm converges to some point of the set $E$.
\label{acumulacao100}
\end{proposicao}
\begin{proof}
As $\left\{x^k\right\}$ is bounded, then there exists a subsequence $\left\{x^{k_j}\right\}$ such that $\lim  \limits_{j\rightarrow +\infty}x^{k_j} = \widehat{x}$. Since $F$ is continuous in $\mathbb{R}^n$, then the function $\left\langle F(.), z\right\rangle$ is also continuous in $\mathbb{R}^n$ for all $z \in \mathbb{R}^m$, in particular, for all $z \in \mathbb{R}^m_+ \backslash \left\{0\right\}$, and
$\left\langle F(\widehat x),z\right\rangle = \lim \limits_{j\rightarrow +\infty}\left\langle F(x^{k_j}) , z\right\rangle $.  On the other hand, we have that $F(x^{k+1}) \preceq F(x^{k})$, and so, $\left\langle F(x^{k+1}) , z\right\rangle \leq \left\langle F(x^{k}) , z\right\rangle $ for all  $z \in \mathbb{R}^m_+ \backslash \left\{0\right\}$.  Furthermore the function $\left\langle F(.), z\right\rangle$  is 
bounded below, for each  $z \in \mathbb{R}^m_+\backslash \left\{0\right\}$, then the sequence $\left\{\left\langle F(x^k),z\right\rangle\right\}$ is  nonincreasing and bounded below, thus convergent. So, 
{\small
\begin{center}
$\left\langle F(\widehat {x}),z\right\rangle = \lim \limits_{j\rightarrow +\infty}\left\langle F(x^{k_j}) , z\right\rangle = \lim \limits_{j\rightarrow +\infty}\left\langle F(x^{k}) , z\right\rangle =  inf_{k\in \mathbb{N}}\left\{\left\langle F(x^k),z\right\rangle\right\}\leq \left\langle F(x^k),z\right\rangle$.
\end{center}
} 
It follows that $F(x^k) - F(\widehat{x}) \in \mathbb{R}^m_+$, i.e, $F(\widehat{x})\preceq F(x^k), \forall \ k  \in \mathbb{N}$. Therefore $\widehat{x}\in E$.  Now, from Proposition $\ref{fejer3}$, we have that the sequence $\{ \left\|\widehat{x} - x^k\right\|\}$ is convergent, and since $\lim  \limits_{k\rightarrow +\infty}\left\|x^{k_j} - \hat{x}\right\| = 0$, we conclude that $\lim  \limits_{k\rightarrow +\infty}\left\|x^{k} - \widehat{x}\right\| = 0$, i.e, $\lim \limits_{k\rightarrow +\infty}x^{k} = \widehat{x}.$
\end{proof}
 \begin{teorema}
 \label{conv4}
Suppose that the assumptions ${\bf(C_{1.2})},$ ${\bf(C_2)}$, ${\bf(C_3)}$ and ${\bf(C_4)}$ are satisfied.  If $0 < \alpha_k < \widetilde{\alpha}$, then the sequence $\lbrace x^k\rbrace$  generated by the {\bf ISPP} algorithm , $(\ref{inicio2})$, $(\ref{diferencial})$ and $(\ref{delta}),$ converges to a Pareto critical point of the problem $(\ref{pom3})$.
\end{teorema}
\begin{proof}
 From Proposition $\ref{acumulacao100}$ there exists $\widehat{x}\in E$ such that  $\lim \limits_{j\rightarrow +\infty}x^{k}= \widehat{x}$.  Furthermore, as the sequence $\left\{z^k\right\}$ is bounded, then there exists $ \left\{z^{k_j}\right\}_{j \in \mathbb{N}}$ such that $\lim \limits_{j\rightarrow +\infty}z^{k_j}= \bar{z}$, with $\bar{z} \in \mathbb{R}^m_+\backslash \left\{0\right\}$.  From $(\ref{diferencial})$ there exists $g_{k_j} \in \hat{\partial}_{\varepsilon_{k_j}}\Psi_{k_j} (x^{{k_j}+1})$, with $g_{k_j} = \nabla\Psi_{k_j}(x^{{k_j}+1}) + \varepsilon_{k_j} h_{k_j}$ with $\Vert h_{k_j} \Vert \leq 1 $, and  $\nu_{k_j} \in \mathcal{N}_{\Omega_{k_j}}(x^{{k_j}+1})$, such that:
 
\begin{equation}
0 = \sum_{i=1}^m \nabla F_i(x^{{k_j}+1})(z_{k_j})_i + \varepsilon_{k_j} h_{k_j} + \alpha_{k_j} \left(x^{{k_j}+1} - x^{k_j} \right) + \nu_{k_j}  
\label{otimalidade2}
\end{equation}
Since $\nu_{k_j} \in \mathcal{N}_{\Omega_{k_j}}(x^{{k_j}+1})$ then, 
\begin{equation}
\left\langle \nu_{k_j} \ ,\ x - x^{{k_j}+1}\right\rangle \leq\  0,\  \forall \ x \in \Omega_{k_j}
\label{cone6}
\end{equation}
Take $\bar{x} \in E$.  By definition of $E$, $\bar{x} \in \Omega_k$, for all $ k \in \mathbb{N}$, so $\bar{x} \in  \Omega_{k_j}$.  Combining $(\ref{cone6})$ with $x = \bar{x}$ and $(\ref{otimalidade2})$, we have
{\footnotesize
\begin{eqnarray}
 0 &\leq &  \left\langle  \sum_{i=1}^m \nabla F_i(x^{{k_j}+1})(z_{k_j})_i\ ,\ \bar{x} - x^{{k_j}+1}\right\rangle + \varepsilon_{k_j}\left\langle h_{k_j}\ ,\ \bar{x} - x^{{k_j}+1}\right\rangle + 
+\alpha_{k_j}\left\langle x^{{k_j}+1} - x^{k_j}\ ,\ \bar{x} - x^{{k_j}+1}\right\rangle \nonumber \\
&\leq & \left\langle\sum_{i=1}^m \nabla F_i(x^{{k_j}+1})(z_{k_j})_i\ ,\ \bar{x} - x^{{k_j}+1}\right\rangle + \varepsilon_{k_j}M + \tilde{\alpha}\left\langle x^{{k_j}+1} - x^{k_j}\ ,\ \bar{x} - x^{{k_j}+1}\right\rangle \nonumber \\
\label{aaa}
\end{eqnarray}
}Observe that, $\forall \ x \in \mathbb{R}^n$:

\begin{eqnarray}
\left\|x^{k+1} - x^k \right\|^2 &=& \left\| x - x^k \right\|^2 - \left\| x - x^{k+1}\right\|^2 + 2\left\langle x^k - x^{k+1} , x - x^{k+1} \right\rangle 
\label{consecutiva}
 \end{eqnarray}
Now, from $(\ref{a})$ with $x = \bar{x} \in E$, and $(\ref{f2})$, we obtain
\begin{eqnarray*}
\left\langle x^k - x^{k+1} ,\bar{x} - x^{k+1} \right\rangle\leq \left\| \bar{x} - x^{k+1}\right\|\left( \dfrac{\varepsilon_k}{\alpha_k} + \dfrac{\Vert \nu_k \Vert}{\alpha_k}\right) \leq 2M\delta_k
\end{eqnarray*}
Thus, from $(\ref{consecutiva})$, with $x = \bar{x}$, we have
\begin{eqnarray}
0 \leq \left\|x^{k+1} - x^k \right\|^2 \leq \left\| \bar{x} - x^k \right\|^2 - \left\| \bar{x} - x^{k+1}\right\|^2 +4M\delta_k
\label{consecutiva2}
\end{eqnarray}
Since that the sequence $\left\{ \left\|\bar{x} - x^k\right\|\right\}$ is convergent and $\displaystyle\sum_{i=1}^{\infty}\delta_k < \infty$, from $(\ref{consecutiva2})$ we conclude that $\displaystyle \lim_{k \to +\infty} \left\|x^{k +1} - x^k \right\| = 0$.  Furthermore, as
\begin{eqnarray}
0 \leq \left\|x^{{k_j}+1} - \bar{x} \right\| \leq \Vert x^{{k_j}+1} - x^{k_j} \Vert + \Vert x^{k_j} - \bar{x}\Vert,
\label{zero}
 \end{eqnarray}
we obtain that the sequence $\{ \left\|\bar{x} - x^{{k_j}+1}\right\|\}$ is bounded.\\
Thus returning to $(\ref{aaa})$, since $\lim \limits_{k\rightarrow+\infty}\varepsilon_k = 0 $, $\lim \limits_{j\rightarrow +\infty}x^{k}= \widehat{x}$ and $\lim \limits_{j\rightarrow +\infty}z^{k_j}= \bar{z}$, taking $j \rightarrow + \infty$, we obtain
\begin{equation}
\sum_{i=1}^m \bar{z}_i\left\langle \nabla F_i(\widehat{x}) \ ,\ \bar{x} -  \widehat{x} \right\rangle\geq 0.
\label{cc}
\end{equation}
Therefore, analogously to the proof of Theorem $\ref{teoparetocri}$, starting in $(\ref{somai0})$, we conclude that $\widehat{x}$ is a Pareto critical point to the problem  $(\ref{pom3})$.
\end{proof}
\section{Finite convergence to a Pareto optimal point}
\noindent

In this section we prove the finite convergence of a particular inexact scalarization proximal point algorithm for proper lower semicontinuous convex functions, which we call Convex Inexact Scalarization Proximal Point algorithm, {\bf CISPP} algorithm.

Let $F:\mathbb{R}^n\longrightarrow \mathbb{R}^m\cup \{+ \infty \}^m$ be a proper lower semicontinuous convex function and consider $ z \in \mathbb{R}^m_+\backslash \left\{0\right\}$ with $\left\|z\right\| = 1$ and the sequences of the proximal parameters $\left\{\alpha_k\right\}$ such that $0<\alpha_k<\bar{\alpha}.$
\begin{description}
\item [{\bf CISPP algorithm}] 

\item [\bf Initialization:] Choose  an arbitrary starting point
\begin{eqnarray}
x^0\in\mathbb{R}^n 
\label{inicio2c}
\end{eqnarray}
\item [Main Steps:] Given $x^k,$ and find $x^{k+1}$ satisfying
 \begin{equation}
e^k\in \partial \left( \langle F(.), z\rangle+ \frac{\alpha_k}{2}\|. - x^k\|^2 \right)(x^{k+1})
  \label{recursao0f2c}
 \end{equation}
\begin{equation}
\label{deltac}
\displaystyle \sum_{k=1}^{\infty} ||e^k||<+\infty,
\end{equation}
where $\partial$ is the classical subdifferential for convex functions.
\item [Stop criterion:] If $x^{k+1}=x^{k} $ or $x^{k+1}$ is a Pareto optimal point, then stop.                                           Otherwise to do $k \leftarrow k + 1 $ and return to Main Steps.
\end{description}
\begin{teorema}
Let $F:\mathbb{R}^n\longrightarrow \mathbb{R}^m\cup \{+ \infty \}^m$ be a proper lower semicontinuous convex function, $0\preceq F$ and assume that $\{x^k\}$ is a sequence generated by the {\bf CISPP} algorithm, $(\ref{inicio2c})$, $(\ref{recursao0f2c})$ and $(\ref{deltac})$. Consider also that the set of Pareto optimal points of $(\ref{pom3}),$ denoted by $Min(F),$ is nonempty and assume that $Min(F)$ is $W_{F(\bar x)}$-weak sharp minimum for the problem $(\ref{pom3})$ with constant $\tau>0$ for some $\bar x\in Min(F).$ Then the sequence $\{x^k\}$ converges, in a finite number of iterations, to a Pareto optimal point.
\end{teorema}
\begin{proof} 
Denote by $g(x)=\langle F(.), z\rangle$ and $$U=\textnormal{arg min}\{g(x): x\in \mathbb{R}^n\}.$$
As $Min(F)$ is nonempty and it is  $W_{F(\bar{x})}$-weak sharp minimum then $Min(F)=WMin(F),$ where $WMin (F)$ denotes the weak Pareto solution of the problem (\ref{pom3}). From Theorem 4.2 of \cite{Bento} it follows that $U$ is nonempty.\\
On the other hand, it is well known that the above {\bf CISPP} algorithm is well defined and converges to some point of $U,$ see 
Rockafellar \cite{Rocka}. We will prove that this convergence is obtained to Pareto optimal point in a finite number of iterations.\\
Suppose, by contradiction, that the sequence $\{x^k\}$ is infinite and take $x^*\in U.$ 
From the iteration $(\ref{recursao0f2c})$ we have that
$$
g(x^{k+1})-g(x^*)\leq \frac{\alpha_k}{2} \left( ||x^k-x^*||^2-||x^{k+1}-x^k||^2\right)+||e^k|||x^{k+1}-x^*||
$$
From (\ref{norma2}) the above inequality implies
$$
g(x^{k+1})-g(x^*)\leq \frac{\alpha_k}{2} \left( ||x^{k+1}-x^*||^2 + 2||x^{k+1}-x^k||||x^{k+1}-x^*||\right)+||e^k|||x^{k+1}-x^*||
$$
Taking $x^*\in U$ such that $||x^{k+1}-x^*||=d(x^{k+1}, W_{F(\bar x)})$ and using the condition of $0<\alpha_k<\bar{\alpha},$ we obtain
$$
\frac{2\tau}{\bar{\alpha}}\leq d(x^{k+1},W_{F(\bar x)})+2||x^{k+1}-x^k|| +\frac{2}{\bar{\alpha}}||e^k||
$$
Letting $k$ goes to infinite in the above inequality we obtain that
$$
\frac{2\tau}{\bar{\alpha}}\leq 0,
$$
which is a contradiction. Thus the {\bf CISPP} algorithm converges to a some point $\hat{x}\in U$ in a finite number of steps.\\
Finally, we will prove that the point of convergence of $\{x^k\},$ denoted by $\hat{x}\in U,$ is a Pareto optimal point of the problem (\ref{pom3}). In fact, as $g$ is weak scalar of the vector function $F,$ then from Proposition \ref{inclusao}, we have $\hat{x}\in WMin(F)$ anf from the equality $Min(F)=WMin(F),$ we obtain that $\hat{x}\in U,$ is a Pareto optimal point of the problem.
\end{proof}

\section{A Numerical Result}
\noindent

In this subsection we give a simple numerical example showing the functionally of the proposed method. For that we use a Intel Core i5 computer 2.30 GHz, 3GB of RAM, Windows 7 as operational system with SP1 64 bits and we implement our code using MATLAB software 7.10 (R2010a).

\begin{exem}
Consider the following multiobjective minimization problem
$$
\min \left\{(F_1(x_1,x_2),F_2(x_1,x_2)): (x_1,x_2)\in \R^2 \right\}
$$
where $F_1(x_1,x_2)=-e^{-x_1^2-x_2^2}+1$ and $F_2(x_1,x_2)=(x_1-1)^2+(x_2-2)^2.$ This problem satisfies the assumptions $\bf (C_{1.2}),$ $\bf (C_2)$ and $\bf (C_4).$ We can easily verify that the points $\bar x=(0,0)$ and $\hat{x}=(1,2)$ are Pareto solutions of the problem.\\
We take $x^0=(-1,3)$ as an initial point and given $x^k\in \R^2,$ the main step of the {\bf SPP} $\mbox{algorithm}$ is to find a critical point ( local minimum, local maximum or a saddle point) of the following problem
%
$$
\label{e}
\left\{\begin{array}{l}
		\min g(x_1,x_2)=(-e^{-x_1^2-x_2^2}+1)z_1^k + \left((x_1-1)^2+(x_2-2)^2 \right)z_2^k+\frac{\alpha_k}{2}\left((x_1-x_1^k)^2+(x_2-x_2^k)^2\right)
\\
		s.to:\\
		\hspace{0.5cm}  x_1^2+x_2^2\leq (x_{1}^k)^2+(x_{2}^k)^2\\
		\hspace{0.5cm}  (x_1-1)^2+(x_2-2)^2\leq (x_1^k-1)^2+(x_2^k-2)^2
	\end{array}\right.
	$$
In this example we consider $z_k=\left(z_1^k,z_2^k\right)=\left(\frac{1}{\sqrt{2}},\frac{1}{\sqrt{2}}\right)$ and $\alpha_k=1,$ for each $k.$
We take $z^0=(2,3)$ as the initial point to solve all the subproblems using the MATLAB function fmincon (with interior point algorithm) and we consider the stop criterion $||x^{k+1}-x^k||<0.0001$ to finish the algorithm. The numerical results are given in the following table:
{\scriptsize	$$ \begin{tabular}{|c|c|c|c|c|c|c|c|}\hline      
    $k$& $  N[x^{k}] $& $x^k=(x^{k}_{1}, x^{k}_{2}) $ &$ ||x^{k}-x^{k-1} || $& $\sum F_i(x^k)z_i^k$& $ F_1(x_1^k,x_2^k)$ & $F_2(x_1^k,x_2^k)$ \\ \hline     
 1  & 10  & (0.17128, 2.41010)& 1.31144& 1.30959& 0.99709 &0.85496      \\
 2  &  10 & (0.65440, 2.16217) &0.54302  & 0.80586 & 0.99392 & 0.14574 	 \\
 3  & 9 &   (0.85337, 2.05877 ) &0.22423 &  0.71983 & 0.99303     &	 0.02496  \\
 4 &  7  &  (0.93534, 2.01588 ) & 0.09251 & 0.70518 & 0.99284     &	 0.00443    \\
 5 & 7  &  (0.96912, 1.99814) & 0.03816 & 0.70268 & 0.99279      &	 0.00096     \\
 6 & 7 &   (0.98305, 1.99080)  & 0.01574 &  0.70226 &  0.99277   &	 0.00037    \\
 7 &  7 &  (0.98879, 1.98776)  & 0.00649 &  0.70219 & 0.99277     &	 0.00028     \\
 8 &  7  &  (0.99115,1.98651) &  0.00268 & 0.70217 & 0.99276     &   0.00026   \\
 9 & 7 &   (0.99213, 1.98599)  & 0.00110 &  0.70217& 0.99276      &  0.00026   \\
 10 & 7 &    (0.99253, 1.98578) & 0.00046 & 0.70217 & 0.99276      &  0.00026  \\ 
 11 & 7 &    (0.99270, 1.98569) &  0.00019 & 0.70217 & 0.99276      &  0.00026  \\ 

 12 & 7 &    (0.99277,1.98565) &  0.00008 & 0.70217 & 0.99276      &  0.00026   \\ \hline
  \end{tabular} 
    $$}
The above table show that we need $k=12$ iterations to solve the problem, $N[x^{k}]$ denotes the inner iterations of each subproblem to obtain the point $x^{k},$ for example to obtain the point $x^3=(0.85337, 2.05877 )$ we need $N[x^3]=9$ inner iterations. Observe also that in each iteration we obtain $F(x^k)\succeq F(x^{k+1})$ and the function $\langle F(x^k),z^k\rangle$ is non increasing.
\end{exem}

\section{Conclusion}
\noindent

This paper introduce an exact linear scalarization proximal point algorithm, denoted by {\bf SPP} algorithm, to solve arbitrary extended multiobjective quasiconvex minimization problems. In the differentiable case it is presented an inexact version of the proposed algorithm and for the (not necessary differentiable) convex case, we present an inexact algorithm and we introduced some conditions to obtain finite convergence to a Pareto optimal point.

To reduce considerably the computational cost in each iteration of the {\bf SPP} algorithm it is need to consider the unconstrained iteration
\begin{equation}
\label{subdiferencialintF}
0 \in \hat{\partial}\left( \left\langle F(.), z_k\right\rangle  + \dfrac{\alpha_k}{2} \Vert\ .\  - x^k \Vert ^2 \right) (x^{k+1}) 
\end{equation}
which is more practical than (\ref{subdiferencial3}). One natural condition to obtain (\ref{subdiferencialintF}) is that $x ^{k +1} \in (\Omega_k)^0$ (interior of $\Omega_k$). So we believe that a variant of the {\bf SPP} algorithm may be an interior variable metric proximal point method.

A future research may be the extension of the proposed algorithm for more general constrained vector minimization problems using proximal distances. Another future research may be to obtain a finite convergence of the {\bf SPP} algorithm for the quasiconvex case.
%


\end{document}